%% file: main.tex
\documentclass[10pt, reqno]{amsart}
\usepackage{amssymb,verbatim,amscd,amsmath}

\usepackage{braket}
\usepackage{array}
\usepackage{mathtools}
\usepackage{verbatim}
\usepackage{caption}
\usepackage{subcaption}
\usepackage{pdfsync}
\usepackage{tabu}
\usepackage{fullpage}
\usepackage{color}
\usepackage[svgnames]{xcolor}
\usepackage{enumerate}
\usepackage{pst-node,pst-text,pst-3d,pstricks}
\usepackage[dvips,all]{xy}
\usepackage{graphics}
\usepackage{tikz}
\usepackage{pst-plot}
\usepackage{cleveref}

\usetikzlibrary{patterns}
\usetikzlibrary{calc}
\usetikzlibrary{matrix}
\usetikzlibrary{arrows}
\usetikzlibrary{decorations.pathmorphing}
\usepackage{graphicx,supertabular,paralist}
\usetikzlibrary{decorations.pathreplacing, decorations.markings}
\tikzset{->-/.style={decoration={markings,mark=at position #1 with {\arrow{>}}},postaction={decorate}}}
\usepgflibrary{arrows}

\numberwithin{equation}{section}
\newtheorem{theorem}{Theorem}[section]
\newtheorem{lemma}[theorem]{Lemma}
\newtheorem{proposition}[theorem]{Proposition}
\newtheorem{corollary}[theorem]{Corollary}

\newtheorem{construction}[theorem]{Construction}

\theoremstyle{definition}
\newtheorem{definition}[theorem]{Definition}

\newtheorem{def-prop}[theorem]{Definition-Proposition}
\newtheorem{remark}[theorem]{Remark}
\newtheorem{example}[theorem]{Example}

\newtheorem*{acknowledgement}{Acknowledgement}
\newtheorem{notation}[theorem]{Notation}

\newtheorem*{Mysketch}{Sketch of proof} 
  {\pushQED{\qed}\begin{Mysketch}}
  {\popQED\end{Mysketch}}
  
  \newtheorem*{Myproof}{Proof of the Claim} 
  {\pushQED{\qed}\begin{Myproof}}
  {\popQED\end{Myproof}}

\DeclareMathOperator{\pol}{pol}

\DeclareMathOperator{\reg}{reg}

\DeclareMathOperator{\depth}{depth}
\DeclareMathOperator{\pd}{pd}
\DeclareMathOperator{\supp}{supp}

\newcommand{\G}{\mathcal{G}}

\newcommand{\D}{\mathcal{D}}

\newcommand{\C}{\mathcal{C}}

\newcommand{\mfS}{\mathcal{S}}
\newcommand{\dS}{\dot{\mathcal{S}}}

\newcommand{\BR}{{\Big\rfloor}}
\newcommand{\BL}{{\Big\lfloor}}

\newcommand{\ZZ}{{\mathbb Z}}

\def\G{{\mathcal G}}
\def\P{{\mathcal P}}

\def\G{{\mathcal G}}

\def\a{{\bf a}}

\def\p{{\bf p}}
\def\q{{\bf q}}

\def\1{{\bf 1}}
\def\0{{\bf 0}}

\begin{document}

\title{Algebraic invariants of weighted oriented graphs}

\author{Selvi K. Beyarslan}
\address{Department of Mathematics and Statistics, University of South Alabama, 411 University Boulevard North, Mobile, AL 36688, USA}
\email{selvi@southalabama.edu}
\urladdr{}

\author{Jennifer Biermann}
\address{Department of Mathematics and Computer Science, Hobart and William Smith Colleges \\
300 Pulteney St.
Geneva, NY 14456, USA}
\email{biermann@hws.edu}
\urladdr{}

\author{Kuei-Nuan Lin}
\address{Department of Mathematics,
Penn State University Greater Allegheny\\ 4000 University Dr, McKeesport, PA 15132, USA}
\email{kul20@psu.edu}
\urladdr{}

\author{Augustine O'Keefe}
\address{Department of Mathematics and Statistics, Connecticut College\\
270 Mohegan Avenue Pkwy,
New London, CT 06320, USA}
\email{aokeefe@conncoll.edu}
\urladdr{}


\begin{abstract} 
Let $\D$ be a weighted oriented graph and let $I(\D)$ be its edge ideal in a polynomial ring $R$. We give the formula of Castelnuovo-Mumford regularity of $R/I(\D)$ when $\D$ is a weighted oriented path or cycle such that edges of $\D$ are oriented in one direction. Additionally, we compute the projective dimension for this class of graphs. 
\end{abstract}

\maketitle


\input{intro.tex}
\input{preliminaries.tex}

\input{labeled.tex}

\input{splitting.tex}
\input{path.tex}

\input{Cycle.tex}

\input{references.tex}
\end{document}

%% file: intro.tex
\section{Introduction}

A vertex-weighted (or simply weighted) oriented graph is a triple $\D = (V(\D), E(\D), w)$ where  $V(\D) = \{x_1, \dots, x_n\}$ is the vertex set of the graph, $E(\D)$ is a directed edge set, and $w$ is a weight function $w: V(\D) \to \mathbb{N^+}$. Specifically, $E(\D)$ consists of ordered pairs $(x_i, x_j) \in V(\D)\times V(\D)$ where the pair $(x_i,x_j)$ represents a directed edge from $x_i$ to $x_j$. We consider finite simple oriented graphs, that is graphs in which $V(\D)$ is a finite set and there are neither loops nor multiple edges in $E(\D)$.  Furthermore, we simplify notation by setting  $w_i = w(x_i)$ for each $i\in \{1,\ldots,n\}.$  If the weight of a vertex $x_i$ is equal to one, i.e. $w_i =1,$ we say that the graph has a \emph{trivial weight} at $x_i.$ Otherwise, we say $x_i$ has a \emph{non-trivial weight}.

Given a weighted oriented graph $\D = (V(\D), E(\D), w)$ and $R = k[x_1, \dots, x_n]$ the polynomial ring on the vertex set $V(\D)$ over a field $k$,  the edge ideal of $\D$ is defined to be
	\[
		I(\D) = \left(x_ix_j^{w_j} : (x_i, x_j) \in E(\D)\right) \subseteq R.
	\]
The generators of  $I(\mathcal{D})$ are independent of the weight assigned to a source vertex. Therefore to simplify our formulas, throughout this paper, we shall assume that source vertices always have weight one. Furthermore, since isolated vertices do not contribute to the generating set of the edge ideal of weighted oriented graph, we assume that our graphs have no isolated vertices.
When \emph{all} vertices have weight one, $I(\D)$ is the edge ideal of an unweighted unoriented graph which was introduced by Villarreal in \cite{V} and has since been studied extensively.  The study of edge ideals of weighted oriented graphs is much more recent and consequently there are many fewer results in this direction. The Cohen-Macaulayness of edge ideals of weighted oriented graphs has been studied in \cite{GMSV}, \cite{HLMRV}, \cite{PEJ} and \cite{VV}. Related to our work in this paper, the authors of \cite{ZW} consider the Castelnuovo-Mumford regularity (hereafter referred to as just regularity), and the projective dimension of a special case of weighted oriented paths and cycles.

The interest in studying edge ideals of weighted oriented graphs has its foundation in coding theory, specifically the study of Reed-Muller-type codes (see \cite{MPV}). Such codes arise as the image of a degree $d$ evaluation map of a given set of projective points over a finite field. The regularity of the vanishing ideal provides a threshold for the degree of the map indicating when a Reed-Muller-type code has sufficiently large minimal distance, and is thus considered ``good." The vanishing ideal is itself a binomial ideal whose initial ideal is exactly the edge ideal of an appropriately defined weighted oriented graph. Although perhaps a weaker bound, the regularity of the vertex-weighted oriented graph is easier to compute and, therefore, provides valuable information in eliminating ``bad" Reed-Muller-type codes.

In this paper we study the regularity and projective dimension of edge ideals of weighted oriented graphs with the goal of characterizing these algebraic invariants in terms of the combinatorial data of our weighted oriented graphs. To describe the generators of the edge ideal of a weighted oriented graph one needs to consider the structure of the underlying undirected graph, the orientation of its edges, and its weight-function. Thus it is quite a difficult problem to incorporate all of this data and provide general formulas for the regularity and the projective dimension of an arbitrary weighted oriented graph. As a natural first step we consider two basic structures: paths and cycles.  We further restrict our attention to weighted oriented paths and cycles with the natural orientation of all edges pointing in the same direction. We call these graphs \emph{weighted naturally oriented} paths and cycles. 

The main results of this paper provide formulas for the regularity of  weighted  naturally oriented paths (\Cref{path:gen}) and  weighted naturally oriented cycles (\Cref{cycle:gen}).  Our results place no restrictions on the weight-function associated to these graphs apart from the requirement that any source vertices have weight one.  The case in which all non-source vertices have non-trivial weights was considered in \cite{ZW}, the results of which are recovered in this paper.

While the majority of this paper focuses on the regularity of these edge ideals, one can also compute their projective dimension by viewing them as the ideals of string hypergraphs and cycle hypergraphs as studied in \cite{LMa1}. We translate the necessary notions from \cite{LMa1} in terms of weighted oriented paths and cycles and present the formulas for their projective dimensions (\Cref{pdPathLMa} and \Cref{pdCycleLinMa}), thus completing the discussion on the projective dimension of weighted naturally oriented paths and cycles where all non-source vertices can have arbitrary weight.

Our paper is organized as follows. In \Cref{pre}, we collect necessary terminology and results from the literature. In \Cref{hyper} we discuss the connection between edge ideals of weighted oriented graphs and labeled hypergraphs. In particular, we present the formulas for the projective dimension of weighted naturally oriented paths and cycles. \Cref{splitting} focuses on Betti splittings of monomial ideals in which we give a large class of Betti splittings of edge ideals of weighted oriented graphs to be used in later sections. In \Cref{paths} we introduce our first main result \Cref{path:gen}, which gives a formula for the regularity of weighted naturally oriented paths.  In \Cref{cycles} we extend this result to weighted naturally oriented cycles in \Cref{cycle:gen}.

\begin{acknowledgement} 
This material is based upon work supported by the National Security Agency under Grant No. H98230-19-1-0119, The Lyda Hill Foundation, The McGovern Foundation, and Microsoft Research, while the authors were in residence at the Mathematical Sciences Research Institute in Berkeley, California, during the summer of 2019. The authors would like to thank Mathematical Sciences Research Institute for their support and hospitality.
\end{acknowledgement}

%% file: preliminaries.tex
\section{Preliminaries}\label{pre}
In this section, we collect the notation and terminology that will be used throughout the paper. We follow the convention of the standard text \cite{Villarreal}.

Let $R = k[x_1, \dots, x_n]$ be a polynomial ring over a field $k$ and $M$ be a finitely generated $R$ module.  
Then the minimal free resolution of $M$ is an exact sequence of the form 
	$$
 0\longrightarrow \bigoplus_{j \in \ZZ} R(-j)^{\beta_{p,j}(M)} \longrightarrow \bigoplus_{j \in \ZZ} R(-j)^{\beta_{p-1,j}(M)} \cdots \longrightarrow 
\bigoplus_{j \in \ZZ} R(-j)^{\beta_{0,j}(M)} \longrightarrow M \longrightarrow 0.$$ 

Since the minimal free resolution of a module is unique up to isomorphism, the exponents $\beta_{i,j}(M)$ are invariants of the module called the Betti numbers of $M$. In general, computing Betti numbers explicitly is intractable so we focus instead on coarser invariants which measure the complexity of the module. In particular, this paper focuses on studying the (Castelnuovo-Mumford) regularity and projective dimension of $M=R/I$ where $I$ is an ideal of $R$.  These invariants are defined as follows 
	$$\reg (M) = \max \{ j-i : \beta_{i,j} (M) \neq 0\}$$
	and 
	$$\pd (M) = \max \{ i : \beta_{i,j} (M) \neq 0\}.$$
Calculating or even estimating the regularity or projective dimension for a general ideal is a difficult problem. Thus we restrict our focus to edge ideals of vertex-weighted oriented graphs where we can exploit the combinatorial structure of the graph to give us information about the regularity and the projective dimension of the associated ideal.

Let $I$ be a homogeneous ideal in $R$ and $m$ be a monomial of degree $d$ with the following standard short exact sequence 
\begin{eqnarray}\label{eq:ses}
0 \rightarrow R/(I:m) (-d) \rightarrow R/I \rightarrow R/(I,m) \rightarrow 0.
\end{eqnarray}
One can then obtain the following well-known regularity relationships (see, for example \cite[Section 2.18] {Peeva}).
\begin{lemma}\label{reg}  Let $I$ be a homogeneous ideal in a polynomial ring $R = k[x_1, \dots, x_n]$ and let $m$ be a monomial of degree $d.$ Then
\begin{eqnarray*}
  \reg (R/I) &\leq& \max \{ \reg (R/(I:m))+d, \reg (R/(I,m))\} \\
  \reg (R/(I:m)) +d& \le &\max  \{ \reg (R/I) , \reg (R/(I,m))+1\}\\
 \reg (R/(I,m)) &\le &\max  \{ \reg (R/I) , \reg (R/(I:m)+d-1) \}. 
\end{eqnarray*}
\end{lemma}

When $m$ is a variable, we have the following special property regarding the regularity.

\begin{lemma}{\cite[Lemma 2.10]{DHS}}\label{oneVar}
Let $I$ be a monomial ideal, and $x$ is a variable in $\supp(I).$ Then $\reg  (R/(I,x)) \le \reg (R/I)$ and $\reg (R/I) \in \{\reg (R/(I:x))+1, \reg (R/(I,x))\}$.
\end{lemma}

When discussing algebraic invariants of the edge ideal of a weighted oriented graph $\D$ we simplify notation and use $\reg (\D)$ (respectively $\pd(\D)$) to refer to $\reg (R/I(\D))$ (respectively $\pd(R/I(\D)).$

The following well-known results will be used throughout the paper. See \cite[Lemma 3.2]{HT} for the proof.

\begin{lemma}\label{lem:disjoint} 
Let $R_1 = k[x_1, \ldots , x_n], R_2= k[y_1, \ldots , y_m]$ be polynomial rings over disjoint variables and $R=k[x_1\ldots, x_n, y_1, \ldots, y_m].$ Suppose $I \subset R_1$ and $J\subset R_2$ be two nonzero homogeneous ideals. Then
\begin{enumerate}
\item $\reg (R/(I +J)) = \reg (R_1/I)+\reg (R_2/J),$
\item $\reg (R/(IJ)) = \reg (R_1/I) +\reg (R_2/J)+1.$

\end{enumerate}
\end{lemma}

\begin{remark}\label{rem:Ix}
    There are situations where we want to consider the regularity of the quotient of ideals of the form $(I(\D),x)$ where $x\not\in\supp(I(\D))$. The above lemma yields the equality $\reg(R/(I(\D),x))=\reg(R/I(\D))$. 
\end{remark}

In our proofs, we often deal with the underlying graphs of weighted oriented paths and cycles. The regularity of the edge ideals of these underlying graphs is known and we recall the formulas of regularity for paths and cycles below.

\begin{theorem}\label{rem:pc} \cite[Corollary 7.6.28 and 7.7.34]{J}
Let $P_n$ denote a path on $n$ vertices and $C_n$ denote a cycle on $n$ vertices. Then $\reg (P_n) = \reg(C_n) = \BL{\frac{n+1}{3}}\BR.$
\end{theorem}

A useful tool in the study of monomial ideals is the process of polarization which allows us to pass from a general monomial ideal to a squarefree monomial. We recall here the definition of polarization (for more information on the subject see for example \cite{Peeva}).  

\begin{definition}{ \cite[Construction 21.7]{Peeva}}
Let $R=k[x_1,\ldots, x_n]$ be a polynomial ring over a field $k.$ Given a $n$-tuple $\a=(a_1, \ldots, a_n) \in \ZZ_{\ge 0}^n,$ let $\bf{x^a}$ denote the monomial $x_1^{a_1}\cdots x_n^{a_n} \in R.$ 
\begin{enumerate}
\item The polarization of $\bf{x^a}$ is defined to be $(x^{\a})^{\pol} ,$ where $(\bullet)^{\pol}$ replaces $x_i^{a_i}$ by a product of distinct variables $\prod_{j=1}^{a_i} x_{i,j}.$ \\
\item Let $I=(\bf{x^{a_1}}, \ldots, \bf{x^{a_r}}) \subseteq$$R$ be a monomial ideal. The \emph{polarization} of $I$ is defined to be the ideal $I^{\pol} = ((\bf{x^{a_1}})^{\pol}, \ldots, (\bf{x^{a_r}})^{\pol})$ in a new polynomial ring $R^{\pol} = k [x_{i,j} ~|~ 1 \le i \le n, 1 \le j \le p_i],$ where $p_i$ is the maximum power of $x_i$ appearing in $ \bf{x^{a_1}}, \ldots, \bf{x^{a_r}}.$
\end{enumerate}
\end{definition}

Note that for any monomial ideals $J$ and $K$ we have $(J+K)^{\pol}=J^{\pol}+K^{\pol}.$ 

Polarization is particularly useful as the polarized ideal shares the same Betti numbers as the original ideal, as stated in the following lemma from \cite{HH}. This allows us to utilize the combinatorial structure of objects associated to the squarefree polarized ideal, such as hypergraphs and simplicial complexes, to characterize the algebraic invariants of the original ideal.  

\begin{lemma}{\cite[Corollary 1.6.3]{HH}}\label{pol}
Let $I \subset R$ be a monomial ideal and $I^{\pol} \subset R^{\pol}$ its polarization. Then
\begin{enumerate}
\item $\beta_{i,j} (R/I) =\beta_{i,j} (R^{\pol}/I^{\pol})$ for all $i,j \ge0,$\\
 \item $\pd (R/I) =\pd (R^{\pol}/I^{\pol})$ and $\reg (R/I) =\reg (R^{\pol}/I^{\pol}).$ 
\end{enumerate}
\end{lemma}

There are several ways to relate a squarefree monomial ideal $I$ with a hypergraph.  The common association is obtained by defining the edges of the hypergraph from the generators of the ideal so that $I$ is the edge ideal of the hypergraph. However, in this paper we study the \emph{labeled} hypergraph associated to a given squarefree monomial ideal $I$ by using the construction introduced in \cite{LMc}. In this special construction generators of the ideal correspond to vertices of the hypergraph, and the edges of the hypergraph correspond to variables which are obtained by the divisibility relations between the minimal generators of the ideal. 
We refer an interested reader to \cite{LMc} for details on labeled hypergraphs.

\begin{construction}{\cite{LMc}}\label{cons:labeled}
Let $I\subseteq R= k[x_1,\ldots ,x_n]$ be a squarefree monomial ideal with minimal monomial generating set $\{f_{1},\cdots,f_{m}\}.$ The \emph{labeled hypergraph} of $I$ is the tuple $H(I) = (V(H), E(H), X(H))$.  The set $V(H)=\{1,\ldots,m\}$ is called the vertex set of $H,$.   The set $E(H)$ is called the edge set of $H(I)$ and is the image of the function $\phi: \{x_1,\ldots, x_n\} \rightarrow 2^V$  defined by $\phi(x_i) =\{j ~:~  x_i \text{ divides } f_j\}$ where $2^V$ represents the power set of $V$. The set  $X(H)=\{ x_i : \phi(x_i) \neq \emptyset\}.$  

The label of an edge $F \in E$ is defined as the collection of variables $ x_i \in \{x_1,\ldots ,x_n\}$ such that $\phi(x_i)=F.$ The number $|X(H)|$ counts the number of labels appearing in $H(I).$
\end{construction}

It should be noted that the underlying hypergraph of $H(I)$ is exactly the \emph{dual hypergraph} (see \cite{Berge}) of the hypergraph whose edge ideal is $I.$ 

\begin{example}\label{ex:hyper}
Let $I=(x_{1}x_{2},x_{2}x_{3}x_{4},x_{3}x_{4}x_{5},x_{5}x_{2}) \subseteq k[x_1, \ldots, x_6].$ Set $f_1 = x_1x_2, f_2 = x_{2}x_{3}x_{4}, f_3 = x_{3}x_{4}x_{5},$ and $f_4 = x_{5}x_{2}$.  Then $V(H)=\{1,2,3,4\},$ $E(H)=\{\{1\},\{1,2,4\},\{2,3\},\{3,4\}\},$ and $X(H)=\{ x_1,x_2,x_3,x_4,x_5\}$. See \Cref{fig:hyper}.
\end{example}

\begin{figure}[h]\caption{The labeled hypergraph of $I = (x_{1}x_{2},x_{2}x_{3}x_{4},x_{3}x_{4}x_{5},x_{5}x_{2})$.}\label{fig:hyper} 
\begin{center}
\includegraphics[scale=1]{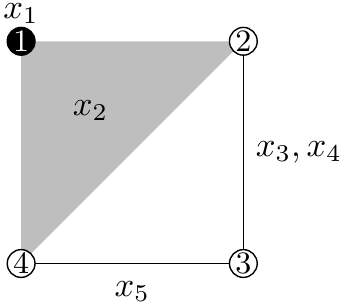}
\end{center}
\end{figure}

Given a weighted oriented graph $\D,$ we can associate the ideal $I^{\pol}(\D)$ with a labeled hypergraph by using \Cref{cons:labeled} and denote this labeled hypergraph by $H(\D).$

%% file: labeled.tex
\section{Labeled hypergraphs associated to edge ideals}\label{hyper}

In this section, we employ the labeled hypergraph construction as the main tool to study the algebraic invariants of $R/I(\D)$ for a weighted oriented graph $\D.$ We collect results from the literature on labeled hypergraphs and provide their immediate applications to our objects. The first part of this section presents formulas for the regularity and projective dimension of weighted oriented paths and cycles with non-trivial weights at all non-source vertices. The second part presents a formula for the projective dimension of \emph{any} weighted naturally oriented path and cycle, thus completing the investigation into projective dimension for these classes of weighted oriented graphs. Furthermore, we provide a corollary that presents a closed formula for the regularity and projective dimension of a large class of weighted oriented graphs.

We begin with the following result from \cite{LMc} which relates the regularity and projective dimension of $R/I(\D)$ to the graph-theoretical invariants of the labeled hypergraph $H(\D).$ Here, we present a translated version of the original statement of \cite[Proposition 4.1]{LMc}.

\begin{proposition}\label{Prop 4.1-LMc}  
Let  $\D$ be a weighted oriented graph with associated labeled hypergraph $H=H(\D)$.  If $\{i\}\in E(H)$ for all $i\in[m]=V(H),$ then $\reg (\D)=|X(H)|-|V(H)|$ and $\pd(\D)=|V(H)|.$
\end{proposition}

Taking $\D$ to be a weighted oriented path or cycle in the above proposition results in explicit expressions for the regularity and projective dimension of the weighted oriented graph. The following corollary is a special case of the above result. Recall that  a vertex $x \in V(\D)$ is called a \emph{leaf}  in $\D$ if there is only one edge incident to $x$, and we assume the source vertices all have weight $1$.  

\begin{corollary}\label{trivial}
Let $\D$ be a weighted oriented graph with weight function $w$ on the vertices $\{x_1,\ldots, x_n\}$ with the property that there is at most one edge oriented into each vertex.  Suppose that for all non-leaf, non-source vertices, $x_j$, either $w_j \geq 2$ or the unique edge $(x_i, x_j)$ into the vertex $x_j$ has the property that $x_i$ is a leaf.  Then
\begin{eqnarray*}
\reg (\D)&=& \sum_{i = 1}^n w_i - |E(\D)|,\\
\pd(\D) &=& |E(\D)|.
\end{eqnarray*}

\end{corollary}

\begin{proof}
Let $I(\D)$ be the edge ideal of the weighted oriented graph $\D$ and let $m_1, \dots, m_r$ be the minimal generators of its polarization $I(\D)^{\text{pol}}$.  For all $1 \leq s \leq r$ let $m_s=x_ix_j\prod_{\ell = 1}^{w_j-1}y_{j, \ell}.$ If $x_j$ is a leaf then $m_s$ is the only minimal generator of $I(\D)^{\text{pol}}$ which is divisible by $x_j$ and therefore $\{s\} \in E(H(\D))$ with label $\{  x_j, y_{j,\ell} \mid 1\leq \ell \leq w_j-1 \}.$ 

On the other hand, if $x_j$ is not a leaf, then by assumption since $x_j$ is not a source, either $w_j \geq 2$ or $x_i$ is a leaf.  If $x_i$ is a leaf then as above $m_s$ is the only minimal generator of $I(D)^{\text{pol}}$ which is divisible by $x_i$ and $\{s\}$ is in $E(H(\D)).$  If $w_j \geq 2$ then $m_s = x_ix_j\prod_{\ell = 1}^{w_j-1}y_{j, \ell}$. In particular $m_s$ is divisible by $y_{j,1}.$  The assumption that there is at most one edge oriented into the vertex $x_j$ means that $y_{j,1}$ does not divide any other generator of $I(\D)^{\text{pol}}$.  Thus in this case, $\{s\} \in E(H(\D))$ with label $\{  y_{j,\ell} \mid 1\leq \ell \leq w_j-1 \}.$ 

Since $\{s\} \in E(H(\D))$ for all $1\leq s \leq r$, by \Cref{Prop 4.1-LMc}
\[
\reg(\D) = |X(H(\D))| - |V(H(\D))| = \sum_{i = 1}^n w_i - |E(\D)|
\]
and 
\[
\pd(\D) = |V(H(\D))| = |E(\D)|.
\]
\end{proof}

As a direct consequence of \Cref{trivial} one can immediately obtain the regularity and projective dimension of a large class of weighted oriented graphs such as naturally oriented paths, naturally oriented cycles, rooted forests, unicyclic graphs with a naturally oriented unique cycle, etc. Thus we recover the results in  \cite{ZW}.

\begin{example}
Let $\P$ be the naturally oriented path on the vertices  $\{x_1, x_2,x_3,x_4\}$ with the edge ideal $I(\P) =(x_1x_2^3,x_2x_3^4,x_3x_4^2).$ Since $\P$ is a path with non-trivial weights at each non-source vertex we have $\reg (\P) =3+4+2-4+2=7\text{ and } \pd (\P)=3$ by \Cref{trivial}. However, if there exists at least one non-source vertex with trivial weight we can not apply the formula given in \Cref{trivial}. For instance, suppose instead that vertex $x_3$ in $\P$ has trivial weight so that $I(\P)=(x_1x_2^3,x_2x_3,x_3x_4^2).$  Computing via Macaulay2 \cite{M2} shows that $\reg(\P)=3$ whereas applying the formula in \Cref{trivial} would give a result of $4$.
\end{example}

In light of the above example, one needs to consider the existence of trivial weights in order to be able to provide a general formula for the regularity and projective dimension of any weighted oriented path and cycle. As an immediate result of applying \Cref{cons:labeled} to weighted oriented paths and cycles, one can see that their associated labeled hypergraphs preserve the path and cycle structures. Following the terminology introduced in \cite{LMa1}, we see that the labeled hypergraph of a weighted oriented path is a string hypergraph and the labeled hypergraph of a weighted oriented cycle is a cycle hypergraph. Since formulas for the projective dimension of string and cycle hypergraphs are given in  \cite{LMa1}, we focus on expressing their results in terms of  weighted oriented graphs and conclude this section by giving a complete picture of the projective dimension of weighted naturally oriented paths and cycles.

The projective dimension formulas  given in \cite{LMa1} use the notion of modularity for string and cycle hypergraphs.  For convenience we translate their definition of modularity into the language of weighted naturally oriented paths and cycles.

\begin{definition}\label{def:modularity}
Let $2 \leq q_1< \ldots <q_k \leq n$ be the positions of non-trivial weights of a weighted naturally oriented path $\P$ on $n$ vertices where $k \geq 1.$  If $q_1 \neq 2$ and/or $q_k \neq n$ then extend the sequence $(q_1, \dots, q_k)$ by appending 2 to its beginning and/or $n$ to its end.  Call this new sequence $p_0, \dots, p_{k+1},$ where this sequence is just equal to the original sequence $q_1,\dots,q_k$ if $q_1=2$ and $q_k=n$.  For each $ i \in \{0, \ldots, k-2\}$ let $\displaystyle \Gamma_{i,\ell}=\Gamma(p_i, \ldots, p_{i+\ell})$ denote the weighted naturally oriented induced path of $\P$ on the vertices $ x_{p_i-1}, \ldots, x_{p_{i+\ell}}$  where $\ell \geq 2$.  

Let $\Gamma$ be the collection of special induced paths defined by
{\small \begin{eqnarray*} \Gamma=\Big\{\Gamma_{i,\ell} ~|~ 0\leq i\leq k-2,\, p_{i+1} -p_i \equiv p_{i+\ell} - p_{i+\ell-1}  \equiv 2 \pmod 3,\, p_{j+1} -p_{j} \equiv 0 \pmod 3  \text{ for } i < j < i+\ell-1 \Big\}.
\end{eqnarray*}}
 Note that for each $i\in\{0,\dots,k-2\}$ there is at most one $\ell\geq 2$ such that $\Gamma_{i,\ell}\in\Gamma$. 

 Thus we may simplify notation and drop the second index $\ell$, $\Gamma_i=\Gamma_{i,\ell}\in\Gamma.$

The \emph{modularity} of $\P$ is the maximal number of induced paths in $\Gamma$ that overlap in at most one edge. Symbolically, we have
    \[
        M(\P)= \max_{i_1<i_2 \cdots < i_t} \Big\{  |\{\Gamma_{i_1}, \ldots, \Gamma_{i_t}\}| ~:~    | E(\Gamma_{i_j}) \cap   E(\Gamma_{i_{j+1}}) | \leq 1 \text{ for }  1 \leq j <t, \Gamma_i \in \Gamma \Big\}.
    \]
It follows from the definition that $M(\P)= 0$ if $\Gamma= \emptyset.$
\end{definition}

\begin{example}

Let $\P_1$ be a weighted naturally oriented path on the vertices $x_1,\ldots, x_9$ with non-trivial weights at vertices $x_4$ and $x_7,$ i.e. $p_1=4$ and $p_2=7.$ By following the convention of \Cref{def:modularity}, we set $p_0=2$ and $p_3=9.$ The only induced path of $\P_1$ satisfying the conditions given in the definition of $\Gamma$ is $\Gamma (2,4,7,9).$ Thus the modularity of $\P_1$ is 1.

Let $\P_2$ be a weighted naturally oriented path on the vertices $x_1,\ldots, x_9$ with non-trivial at vertices $x_3,x_4, x_6,$ i.e. $p_1=3, p_2=4,p_3=6.$ By setting $p_0=2$ and $p_4=9,$ one can see that there exists no induced path of $\P_2$ satisfying the conditions in the definition of $\Gamma.$ Thus $M(\P_2)=0.$

Let $\P_3$ be a weighted naturally oriented path on the vertices $x_1,\ldots, x_{14}$ with non-trivial weights at vertices $x_3, x_5,x_7,x_9,x_{12},$ i.e. $p_1=3, p_2=5, p_3=7, p_4=9, p_5=12.$ By setting $p_0=2$ and $p_6=14,$ one can see that there are two induced paths of $\P_3$ satisfying conditions in the definition of $\Gamma$ and sharing $\{x_6,x_7\}$ as the common edge: $\Gamma(3,5,7)$ and $\Gamma(7,9,12,14).$ Thus $M(\P_3)=2.$

\end{example}

The next result is obtained by rephrasing the statement of \cite[Theorem 3.4] {LMa1} for weighted naturally oriented paths.

\begin{theorem}\label{pdPathLMa}
Let $\mathcal{P}$ be a weighted naturally oriented path on $n$ vertices and $2\leq p_{1}<p_{2}<\cdots<p_{k}\leq n$
be the positions of non-trivial weights in $\mathcal{P}$ where $k\geq 1.$
Then 
\[
\pd(\mathcal{P})=n-1-\sum_{i=0}^{k}\left\lfloor \frac{p_{i+1}-p_{i}+1}{3}\right\rfloor +M(\P).
\]

\end{theorem}

\begin{proof}
Since the associated labeled hypergraph of a weighted oriented path is a string hypergraph on $n-1$ vertices, the statement holds by translating each expression appearing in \cite[Theorem 3.4] {LMa1} to the language of weighted oriented paths.

\end{proof}

In general, the modularity of the weighted naturally oriented cycle is defined similarly to that of the path,in that we are counting the ways we can cover the cycle with induced paths that satisfy certain conditions. We do not fully translate the results from \cite{LMa1} because it is more technical, and we encourage the interested reader to refer to their paper for details. However, it should be noted that there are two main differences between the modularity of the cycle and that of the path. The positions of the non-trivial weights in $\C_n$ can start at any vertex with a non-trivial weight. By reordering the vertices of the cycle, we may assume that $1\leq p_1< \ldots <p_k \leq n$ where $p_1, \ldots, p_k$ are the positions of non-trivial weights in $\C_n.$ Secondly, when we define the cycle analogue to $\Gamma$ in \Cref{def:modularity}, it is possible that one of its elements has overlapping initial and terminal edges. In this case, the modularity would be equal to 1.

Now, we are ready to rephrase the projective dimension formula of cycle hypergraphs from \cite[Theorem 4.3] {LMa1} in terms of weighted naturally oriented cycles.

\begin{theorem} \label{pdCycleLinMa} 
Let $\mathcal{C}_{n}$
be a weighted naturally oriented cycle on $n$ vertices and $1\leq p_{1}<p_{2}<\cdots<p_{k}\leq n$
be the positions of non-trivial weights in $\mathcal{C}_{n}$ where
$k\geq 1$. Then 
\[
\pd(\mathcal{C}_n)=n-\left(\sum_{i=1}^{k-1}\left\lfloor \frac{p_{i+1}-p_{i}+1}{3}\right\rfloor +\left\lfloor \frac{n+p_{1}-p_{k}}{3}\right\rfloor \right)+M(\C_n)
\]
\end{theorem}

\begin{proof}
Since the associated labeled hypergraph of a weighted naturally  oriented cycle is a cycle hypergraph on $n$ vertices, statement holds by \cite[Theorem 4.3] {LMa1}.
\end{proof}

\begin{remark}
As a consequence of the theorems above, one can obtain the depth formulas for weighted naturally oriented paths and cycles using Auslander-Buchsbaum formula. Let $\P$ be a weighted naturally oriented path and $\C_n$ be a weighted naturally oriented cycle on $n$ vertices. Then
\begin{eqnarray*}
\depth (\P) &=& 1+\sum_{i=0}^{k}\left\lfloor \frac{p_{i+1}-p_{i}+1}{3}\right\rfloor -M(\P), \text{ and } \\
\depth (\C_n) &=& \sum_{i=1}^{k-1}\left\lfloor \frac{p_{i+1}-p_{i}+1}{3}\right\rfloor +\left\lfloor \frac{n+p_{1}-p_{k}}{3}\right\rfloor -M(\C_n).
\end{eqnarray*}
\end{remark}

%% file: splitting.tex
\section{Betti Splitting}\label{splitting}

A common strategy used in the study of monomial ideals $I$ in $R$ is to decompose the monomial ideal into smaller ideals and recover the invariants of $I$ using the invariants of the smaller pieces.  Eliahou and Kervaire used this strategy in \cite{EK} when they introduced the notion of a Betti splitting of a monomial ideal.  The idea was developed further by Francisco, H\`a, and Van Tuyl in \cite{FHVT}, where the authors studied when a monomial ideal has a Betti splitting. We recall the definition of this notion along with relevant important results from \cite{FHVT}. We then provide a large class of Betti splittings of edge ideals of weighted oriented graphs to be used in later sections. 

Given a monomial ideal $I$, we denote by $\G(I)$ the set of minimal monomial generators of $I$.

\begin{definition}\label{def:betti}
Let $I, J,$ and $K$ be monomial ideals such that $\G(I)$ is the disjoint union of $\G(J)$ and $\G(K).$ Then
$I = J + K$ is a \emph{Betti splitting} if
\begin{eqnarray*}
\beta_{i,j} (R/I)= \beta_{i,j} (R/J)+\beta_{i,j} (R/K)+\beta_{i-1,j} (R/ J \cap K)
\end{eqnarray*}
for all $i,j \ge 0$ where $ \beta_{i-1,j} (R/ J\cap K) = 0$ if $i=0.$
\end{definition}

\begin{theorem}{\cite[Corollary 2.7]{FHVT}}\label{thm:betti}
Suppose that $I= J + K$ where $\G(J)$ contains
all the generators of $I$ divisible by some variable $x_i$ and $\G(K)$ is a nonempty set
containing the remaining generators of $I.$ If $J$ has a linear resolution, then $I = J+K$
is a Betti splitting.
\end{theorem}

When $ I = J + K$ is a Betti splitting, important homological invariants of $I$ are indeed related to the corresponding invariants of the smaller ideals $J,K.$ The following corollary is a direct consequence of the formulas for the Betti numbers.

\begin{corollary}\label{cor:betti}
Let $I = J + K $ be a Betti splitting. Then 
\begin{eqnarray*}
 \reg (R/I)&=&\max \{ \reg (R/J),\reg (R/K),\reg (R /J \cap K)-1\}\\
 \pd(R/I) &=& \max \{ \pd(R/J), \pd (R/K), \pd (R/ J \cap K)+1\}.
\end{eqnarray*}

\end{corollary}

One can generalize the notion of splitting edge of a graph from \cite{FHVT} to splitting edge of a weighted oriented graph.  Let $e=(x_i,x_j)$ be an edge in a weighted oriented graph $\D.$ If $I(\D)=J+K$ is a Betti splitting when $J=(x_ix_j^{w_j})$ is the  monomial ideal  associated to $e$, the edge $e$ is called \emph{a splitting edge} of $\D.$ The edge $(x_i,x_j)$  in \Cref{splitvertex} is a splitting edge of $\D.$

\begin{proposition}\label{splitvertex}
Let $\D$ be a weighted oriented graph on vertex set $\{x_1 \ldots x_n\}$. Suppose that $w_j >1$ and that $(x_i, x_j)$ is the only edge of $\D$ oriented into the vertex $x_j$.  Let $J = (x_ix_j^{w_j})$ and let $K$ be the ideal generated by $\G(I(\D)) \setminus \{x_ix_j^{w_j}\}$.  Then $I(\D) = J+K$ is a Betti splitting and further
\begin{eqnarray*}
\reg (\D)&=&\max \left\{w_j, \reg (R/K), \reg (R/(J\cap K))-1\right\},\\
\pd (\D)&=&\max \left\{\pd (R/K), \pd (R/(J\cap K))+1\right\}.
\end{eqnarray*}

\end{proposition}
\begin{proof}
Let $I(\D)$ be the edge ideal of $\D$ and let $I(\mathcal{D})^{\text{pol }}$ be its polarization.  The polarization of the generator $x_ix_j^{w_j}$ of $I(\D)$ is $m = x_ix_j \prod_{\ell = 1}^{w_j-1}y_{j,\ell}$.  Since $w_j \geq 2$, $m$ is divisible by $y_{j,1}$.  In particular, since $(x_i, x_j)$ is the only edge which is oriented into $x_j$, $m$ is the only minimal generator of $I(\mathcal{D})^{\text{pol}}$ which is divisible by $y_{j,1}$.  Thus by \Cref{thm:betti}, $I(\D)^{\text{pol}} = J' + K'$ is a Betti splitting where $J' = (m)$ and $K'$ is the ideal generated by $\G(I(\D)^{\text{pol}}) \setminus \{m\}$.  Note that $J' = J^{\text{pol}}$ and $K' = K^{\text{pol}}$ where $J = (x_ix_j^{w_j})$ and $K$ is the ideal generated by $\G(I(\D)) \setminus \{x_ix_j^{w_j}\}$. Furthermore, $(J \cap K )^{\pol} = J' \cap K'$ as both $J$ and $K$ are monomial ideals.  Since polarization preserves Betti numbers by \Cref{pol}, this implies that $I = J+K$ is a Betti splitting. 

The formulas for regularity and projective dimension are a direct application of \Cref{cor:betti} since $\reg(R/J) = w_j$ and $\pd(R/J) = 1.$
\end{proof}

\begin{corollary}\label{cor:bettieq}
Let $\D$ be a weighted naturally oriented path or cycle on vertex set $\left\{x_1 \ldots x_n\right\}$.  Let $x_i$ be a vertex in $\D$ with $w_i>1$. Let $J=(x_{i-1}x_{i}^{w_i})$ and $K$ be the ideal generated by $\G(I(\D)) \setminus \{x_{i-1}x_{i}^{w_i}\}.$  Then $I(\D) = J+K$ is a Betti splitting and 
\[
\reg (\D)=\max \left\{w_i, \reg (R/K), \reg (R/(J\cap K))-1\right\}.
\]
\end{corollary}
\begin{proof}
This follows immediately from \Cref{splitvertex} since  $(x_{i-1},x_i)$ is the only edge  oriented into the vertex $x_i$ in a weighted naturally oriented path or cycle.
\end{proof}

\begin{remark}\label{JL}
If $J\cap K=JL$ for some ideal $L$ such that $J$ and $L$ have different supports, then, by \Cref{lem:disjoint}, $\reg(R/J\cap K)-1=\reg(R/J)+\reg(R/L).$ 
\end{remark}

%% file: path.tex
\section{Weighted Oriented Paths}\label{paths}

In this section, we focus on computing the regularity of weighted naturally oriented  paths.  If the weight of each vertex is trivial, the edge ideal of a weighted naturally oriented path is the same as the edge ideal of an unweighted path and its regularity is given in \Cref{rem:pc}.  At the other extreme, if the weight of each vertex is non-trivial, the regularity and projective dimension of a path $\P$ can be computed explicitly via the labeled hypergraph of $\P$ as observed in \Cref{trivial}. Our main result in this section is \Cref{path:gen} which gives a formula for the regularity of a  weighted naturally oriented path with any combination of trivial and non-trivial weights.  This formula depends both on the weights of the vertices and the distances between successive non-trivial weights.  

Before proving our main result we introduce two lemmas which give the regularity of naturally oriented paths with special arrangements of non-trivial weights. 

\begin{lemma}\label{prop:path1}
Let $\P$ denote the weighted naturally oriented path on $n$ vertices $x_1,\ldots, x_n$ such that $w_p=w \ge 2$ for some $p \in \left\{2,\ldots, n\right\}$ and $w_i=1$ for $i\neq p,$ $1 \le i \le n.$ Then 
    \[
        \reg (\P)= w+\BL{\frac{n-p}{3}\BR}+\BL{\frac{p+1}{3}\BR}-1.
    \]
\end{lemma}

\begin{proof}
By \Cref{cor:bettieq}, $J=(x_{p-1}x_p^{w})$ and $K=(x_{1}x_{2},\ldots,x_{p-2}x_{p-1},x_{p}x_{p+1},\ldots,x_{n-1}x_{n})$ is a Betti splitting of $I(\P)$.  We can view the ideal
$K$ as the edge ideal of the disjoint union of two paths with trivial weights. Thus by \Cref{rem:pc} we have
    \[
        \reg(R/K)=\BL{\frac{p-1+1}{3}\BR} +\BL{\frac{n-(p-1)+1}{3}\BR} =\BL{\frac{p}{3}\BR} +\BL{\frac{n-p+2}{3}\BR}.
    \]
Let $L=(x_{1}x_{2},\ldots,x_{p-4}x_{p-3},x_{p-2},x_{p+1},x_{p+2}x_{p+3},\ldots,x_{n-1}x_{n})$, then $J\cap K= JL$. Since $L$ can be viewed as the sum of an edge ideal of the disjoint union of two paths and an ideal generated by variables not in those paths, by \Cref{rem:Ix} we have
    \[
        \reg(R/L) =\BL{\frac{p-3+1}{3}\BR} +\BL{\frac{n-(p+1)+1}{3}\BR}  =\BL{\frac{p-2}{3}\BR} +\BL{\frac{n-p}{3}\BR}.
    \]
Thus by \Cref{JL}, $\displaystyle  \reg (R/(J\cap K))-1= w+\BL{\frac{p-2}{3}\BR} +\BL{\frac{n-p}{3}\BR}$ since  $\reg (R/J)=w.$  Observe that 
    \[
        \reg(R/J\cap K)-1  =w+\BL{\frac{p-2}{3}\BR} +\BL{\frac{n-p}{3}\BR} \geq \BL{\frac{p}{3}\BR} +\BL{\frac{n-p+2}{3}\BR} =\reg(R/K).
    \]
Therefore, $\displaystyle \reg (\P) =\reg(R/J\cap K)-1 =w+\BL{\frac{p-2}{3}\BR} +\BL{\frac{n-p}{3}\BR}$ by  \Cref{cor:bettieq}. Since $\displaystyle  w+\BL{\frac{p-2}{3}\BR} +\BL{\frac{n-p}{3}\BR}=w+\BL{\frac{p+1}{3}\BR} +\BL{\frac{n-p}{3}\BR}-1$, the lemma is proved.

\end{proof}

\begin{lemma}\label{pathWeights}
Let $\P$ be a weighted naturally oriented path on $n$-vertices with $w_1 = w_2 = \dots = w_{\ell-1} = 1$ and $w_j>1$ for all $\ell \leq j \leq n-1$.  Then
    \[
        \reg(\P) = \BL \frac{\ell+1}{3}\BR + \sum_{i=\ell}^n w_i-(n-l+1).
    \]

\end{lemma}

\begin{proof}
The edge ideal of $\P$ is $I(\P) = (x_1x_2, \dots, x_{\ell-2}x_{\ell-1}, x_{\ell -1}x_\ell^{w_\ell}, x_{\ell}x_{\ell+1}^{w_{\ell+1}}, \dots, x_{n-1}x_n^{w_n})$. Let $J=(x_{\ell-1}x_{\ell}^{w_{\ell}})$ and
and $K=(x_1x_2, \ldots, x_{\ell-2}x_{\ell-1}, x_{\ell}x_{\ell+1}^{w_{\ell+1}}, \ldots, x_{n-1}x_n^{w_n}).$  By \Cref{cor:bettieq}, one can see that $I(\P) = J + K$ is a Betti splitting and therefore $\reg (\P) = \max\{\reg(R/J), \reg(R/K), \reg(R/(J\cap K))-1\}$.  

The ideal $K$ is the sum of the edge ideal of a path on $(\ell-1)$ vertices with trivial weights and a weighted naturally oriented path on $(n-\ell+1)$ vertices where the weights of all non-leaf vertices are non-trivial. Then, by \Cref{lem:disjoint}, \Cref{rem:pc}, and  \Cref{trivial}, we get
    \[
        \reg (R/K)= \BL{\frac{\ell}{3}\BR} + 1 + \sum_{i=\ell+1}^n w_{i}- (n-\ell).
    \]

Let $ L= (x_1x_2, \ldots,x_{\ell-4}x_{\ell-3}, x_{\ell-2}, x_{\ell+1}^{w_{\ell+1}}, x_{\ell+1}x_{\ell+2}^{w_{\ell+2}}, \ldots, x_{n-1}x_n^{w_n})$, then $J \cap K = JL$. 

The ideal $L$ is the sum of the ideal $L_1=(x_{\ell+1}^{w_{\ell+1}}, x_{\ell+1}x_{\ell+2}^{w_{\ell+2}}, \ldots, x_{n-1}x_n^{w_n}),$ the edge ideal of a path on $(\ell-3)$ vertices with trivial weights which has no variables in common with $L_1$, and a variable $x_{\ell-2}$. Further, the generator $x_{\ell-2}$ does not effect the regularity of $R/L$ by \Cref{rem:Ix}. The ideal $L_1$ has the same polarization as the ideal $(x_{\ell}x_{\ell+1}^{w_{\ell+1}-1}, x_{\ell+1}x_{\ell+2}^{w_{\ell+2}}, \ldots, x_{n-1}x_n^{w_n})$ after a relabeling of the variables.  The regularity of this latter ideal can be computed using  \Cref{trivial} and thus $\reg(R/L_1) = \sum_{i=\ell+1}^n w_{p_i}-(n-\ell) $. It follows from \Cref{lem:disjoint} that$$\reg (R/L)= \BL \frac{\ell-2}{3}\BR + \sum_{i=\ell+1}^n w_{p_i}-(n-\ell).$$

By \Cref{JL}, we have
    \begin{eqnarray*}
        \reg (R/(J \cap K))-1
            &=& 
                w_{\ell}+ \BL{\frac{\ell-2}{3}\BR}+\sum_{i=\ell+1}^n w_i-(n-\ell)\\
            &=&  
                \BL{\frac{\ell-2}{3}\BR}+1-1+ w_\ell + \sum_{i=\ell+1}^n w_i-(n-\ell) \\
            &=&  
                \BL{\frac{\ell+1}{3}\BR}+(w_\ell -1)+\sum_{i=\ell+1}^n w_i-(n-\ell)\\
            &\geq& 
                \BL{\frac{\ell}{3}\BR} + 1 +  \sum_{i=\ell+1}^n w_{i}- (n-\ell) \\
            &=& 
                \reg (R/K)
    \end{eqnarray*}
where the last inequality follows from the assumption that $w_{\ell} \geq 2.$ One can see that $\reg (R/(J\cap K))-1 \geq \reg(R/J) = w_\ell$.  Therefore, $\reg (\P)=\reg (R/(J \cap K))-1  $ by \Cref{cor:bettieq}, and the desired equality holds. 
\end{proof}

Our next result is general in the sense that we consider weighted naturally oriented paths with arbitrary numbers of non-trivial weights.  Both the values of the weights and their positions factor into our formula for the regularity of a path, motivating the following definition.

\begin{definition}\label{p_i}
Let $\P$ be a  weighted naturally oriented path on $n$ vertices and $2\le p_1< p_2\le \ldots <p_k\le n$ be the positions of non-trivial weights in $\P$ for $k\geq 1.$ We call  $p=\left(p_1,\ldots, p_k\right)$ the \emph{weight  sequence} of $\P.$ 
\end{definition}

\begin{notation}\label{not}
\sloppy Let $\P$ be a weighted naturally oriented path on $n$-vertices with the weight sequence $\left(p_1,\ldots, p_k\right)$.  In what follows we will abuse notation and write $(q_1, \dots, q_t) \subseteq (p_1, \dots, p_k)$ to mean that $(q_1, \dots, q_t)$ is a subsequence of $(p_1, \dots, p_k)$. 

Let $\mfS$ be the collection of subsequences of the weight sequence $\left(p_1,\ldots, p_k\right)$ where the difference between any consecutive elements of the subsequence is not equal to two, i.e.
    \[
        \mfS= \left\{ \left(q_1,\ldots, q_t\right) \subseteq \left(p_1,\ldots, p_k\right) ~:~   q_{i+1}-q_i \neq 2\text{ for each } i \in \left\{1, \ldots, t-1\right\}\right\}.
    \]

\end{notation}

\begin{example} Let $ \P_1$ be a path on the vertices $x_1,\ldots, x_7$ and $(2,3,5,6)$ be the weight sequence of $\P_1,$ i.e. $w_2,w_3,w_5,w_6 \ge 2$ and $w_4,w_7=1.$  In what follows, we are interested in the elements of $\mfS$ which are maximal with respect to inclusion. The maximal elements of $\mfS$ are
    \[
        \left\{\left(2,3,6\right), \left(2,5,6\right)\right\}.
    \]
Let $ \P_2$ be a path on the vertices $x_1,\ldots, x_7$ and $(2,4,6,7)$ be the weight sequence of $\P_2,$ i.e. $w_2,w_4,w_6,w_7 \ge 2$ and $w_3,w_5=1.$ Then the maximal elements of $\mfS$ are
    \[
        \left\{ \left(2,6,7\right),\left(4,7\right)\right\}.
    \]
\end{example}

As the previous example illustrates, the maximal elements of the set $\mfS$ always begin with $p_1$ or $p_1+2$ with the latter only occuring if $p_2 = p_1+2$.

\begin{definition}\label{not2}
For a  weighted naturally oriented path $\P$ on $n$-vertices with the weight sequence $(p_1,\ldots, p_k),$ we define  \emph{weight-position sum} of $\textbf{q}=\left (q_1,\ldots, q_t\right) \in \mfS$ as follows 
    \[
        \sum {\bf q} = \sum(q_1,\ldots, q_t) :=  \sum_{i=1}^t w_{q_i} + \sum_{i=1}^{t-1} {\BL{\frac{q_{i+1}-q_i}{3}}\BR} -t.
    \]
\end{definition}

We will be interested in the largest element of the set $\left\{ \sum \q+{\BL{\frac{n-q_t}{3}}\BR}+{\BL{\frac{q_1+1}{3}}\BR} \mid \q \in \mfS \right\}$.  To simplify the proof of our main result \Cref{path:gen} we present the following lemma which states that the maximum element of this set always occurs at a maximal (with respect to inclusion) element of $\mfS$.

\begin{lemma}\label{maximal}
If $\P$ is a weighted naturally oriented path with weight sequence ${\bf p}$ and ${\bf q}\subseteq {\bf q'}$ are elements of $\mfS,$ then  
\[
    \sum {\bf q} +\BL \frac{n-q_t}{3} \BR + \BL \frac{q_1+1}{3}\BR \leq \sum {\bf q'}+ \BL \frac{n-q'_{t'}}{3} \BR + \BL \frac{q'_1+1}{3}\BR.
\]

Moreover, 
\[
    \max_{\textbf{q} \in \mfS} \left\{ \sum  \q +{\BL{\frac{n-q_t}{3}}\BR}+{\BL{\frac{q_1+1}{3}}\BR}\right\} = \max_{\substack{\textbf{q} \in \mfS \\ q_1 = p_1, p_1+2}} \left\{ \sum \q +{\BL{\frac{n-q_t}{3}}\BR}+{\BL{\frac{q_1+1}{3}}\BR}\right\}.
\]
\end{lemma}
\begin{proof}
First note that if ${\bf q} = (q_1, \dots, q_t) \subseteq {\bf q'} = (q_1, \dots, q_j, r, q_{j+1}, \dots, q_t) \subseteq {\bf p}$ then since $r$ is a position of a  non-trivial weight of $\P$, we have $w_r-1 \geq 1.$ In addition, it follows from the properties of floor functions that  $\BL \frac{q_{j+1}-q_j}{3}\BR \leq \BL \frac{q_{j+1}-r}{3}\BR + \BL \frac{r-q_j}{3}\BR + 1.$ Therefore  \begin{eqnarray*}
        \sum {\bf q} &=& \sum_{i = 1}^t w_{q_i} + \sum_{i =1}^{j-1}\BL \frac{q_{i+1}-q_i}{3}\BR + \BL \frac{q_{j+1}-q_j}{3}\BR + \sum_{i =j+1}^{t-1}\BL \frac{q_{i+1}-q_i}{3}\BR -t \\
        &\leq& \sum_{i = 1}^t w_{q_i} + w_r -1 + \sum_{i =1}^{j-1}\BL \frac{q_{i+1}-q_i}{3}\BR + \BL \frac{q_{j+1}-r}{3}\BR + \BL \frac{r-q_j}{3}\BR + \sum_{i =j+1}^{t-1}\BL \frac{q_{i+1}-q_i}{3}\BR -t \\
        &=& \sum_{i = 1}^t w_{q_i} + w_r + \sum_{i =1}^{j-1}\BL \frac{q_{i+1}-q_i}{3}\BR + \BL \frac{q_{j+1}-r}{3}\BR + \BL \frac{r-q_j}{3}\BR + \sum_{i =j+1}^{t-1}\BL \frac{q_{i+1}-q_i}{3}\BR -(t+1)\\
        &=& \sum {\bf q'} 
    \end{eqnarray*}
    Then, for any such ${\bf q}, {\bf q}' \in \mfS,$ we have
    $$  \sum {\bf q} + \BL \frac{n-q_t}{3} \BR + \BL \frac{q_1+1}{3}\BR \leq \sum {\bf q'} + \BL \frac{n-q_t}{3} \BR + \BL \frac{q_1+1}{3}\BR. $$
    
    Similar arguments can be made when ${\bf q'} = (r, q_1, \dots, q_t)$ and when ${\bf q'} = (q_1, \dots, q_t, r)$.
    
Extending this idea we can see that if ${\bf q}=(q_1,\ldots, q_t)$ and ${\bf q'}=(q'_1,\ldots, q'_s)$ are any elements of $\mfS$ with ${\bf q\subseteq q'}$ then 
    \[
    \sum {\bf q} +\BL \frac{n-q_t}{3} \BR + \BL \frac{q_1+1}{3}\BR \leq \sum {\bf q'}+ \BL \frac{n-q'_s}{3} \BR + \BL \frac{q'_1+1}{3}\BR.
    \] 
   The final part of the lemma follows from the above inequality and the observation that the maximal elements of the set $\mfS$ begin with $p_1$ or $p_1+2$ where the latter occurs if $p_2=p_1+2.$

    \end{proof}

\begin{example}\label{ex:maximal}
Let $\P_8^1$ and $\P_8^2$ be as shown in Figure \ref{fig:PathMaximal}. 

    \begin{figure}[ht]
        \centering
        \includegraphics{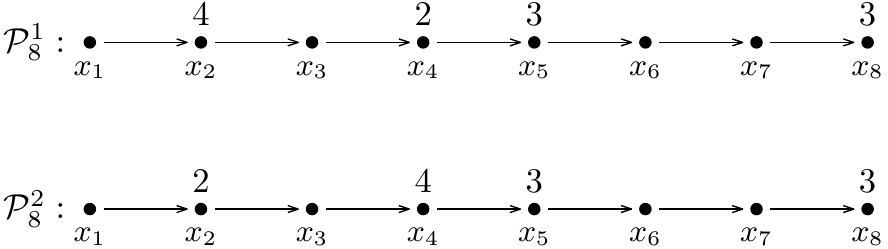}
        \caption{Same weight sequence, different weight functions.}
        \label{fig:PathMaximal}
    \end{figure}

These graphs have the same weight sequence $(2,4,5,8)$ and the same total weight sum. Note that since the two graphs have the same weight sequence, the set $\mfS$ as defined in $\Cref{not}$ will be the same for both graphs.

Given $\q\in\mfS$ define $f(\q)=\sum\q+\BL\frac{n-q_t}{3}\BR+\BL\frac{q_1+1}{3}\BR,$ the expression introduced in \Cref{maximal}. For each of the graphs $\P_8^1$ and $\P_8^2$, the following tables provides the value of $f(\q)$ for each maximal element $\q\in\mfS$.

    \[  
        \P_8^1: ~\begin{array}{r|cc}
            \q & (2,5,8) & (4,5,8)  \\ \hline
            f(\q) & 10 & 7
        \end{array}\hspace{1cm}
        \P_8^2: ~\begin{array}{r|cc}
            \q & (2,5,8) & (4,5,8)\\ \hline
            f(\q) & 8 & 9
        \end{array}
    \]
    
Here we see that the maximum of value of $\sum\q+\BL\frac{n-q_t}{3}\BR+\BL\frac{q_1+1}{3}\BR$ could come from $q_1=p_1=2$ as in $\P_8^1$ or from $q_1=p_1+2=4$ as in $\P_8^2$.  

Furthermore, using Macaulay 2 \cite{M2}, we see that $\reg(\P_8^1)=10$ and $\reg(\P_8^2)=9$, providing evidence that the regularity of the graphs is given by $f(\q)$ which, in turn, depends on both the positions of weighted sequence and the values of the weights.

\end{example}

\begin{theorem}\label{path:gen}
Let $\P$ be a weighted naturally oriented path on $n$-vertices $x_1,\ldots, x_n$  with the weight sequence $(p_1,\ldots, p_k).$ Then
    \[
    \reg (\P) = \max_{\textbf{q} \in \mfS} \left\{ \sum (q_1,\ldots, q_t)+{\BL{\frac{n-q_t}{3}}\BR}+{\BL{\frac{q_1+1}{3}}\BR}\right\}.
    \]
\end{theorem}

Before we proceed to the proof of the main theorem, we present examples to show that the formula of regularity depends not only on the weights of the vertices but also on the weight sequence.

\begin{example}
Let $\P_{12}^1$ and $\P_{12}^2$ as shown in Figure \ref{fig:PathExample2}. These two paths each have four vertices with the same weight values $5,4,3,4,$ but at different positions. Their weight sequences are $(2,5,9,12)$ and $(2,4,8,12),$ respectively. Define $f(\q)=f(q_1,\dots,q_t)$ as in \Cref{ex:maximal}.\\
Using Macaulay 2 \cite{M2}, one can calculate $\reg(\P_{12}^1)=16$ and $\reg(\P_{12}^2)=13$. We can see that $16=f(2,5,9,12)$, i.e. the regularity for $\P_{12}^1$ comes from the entire weight sequence. On the other hand, regularity for $\P_{12}^2$ is given by $13= f(2,8,12)$ where $(2,8,12)$ and  $(4,8,12)$ are the maximal elements of $\mfS$ for $\P_{12}^2$ with values $f(2,8,12) =13\geq f(4,8,12)=11.$
  \begin{figure}[h!]
      \centering
      \includegraphics[scale=1]{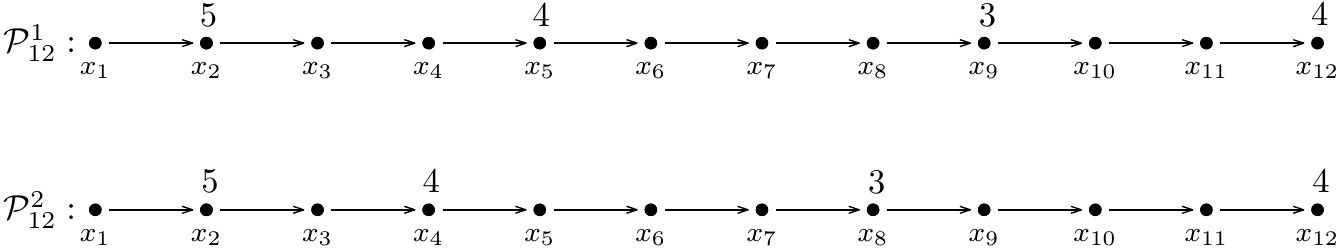}
      \caption{Two weighted oriented paths with the same non-trivial weight values and different weight sequences.}
      \label{fig:PathExample2}
  \end{figure}
  \end{example}

\begin{proof}[Proof of \Cref{path:gen}]

We proceed by induction on the number of non-trivial weights $k.$  The base case $k=1$ is proved in \Cref{prop:path1}. Recall that by \Cref{maximal} that the maximal element of the set $$ \left\{ \sum (q_1,\ldots, q_t)+{\BL{\frac{n-q_t}{3}}\BR}+{\BL{\frac{q_1+1}{3}}\BR}\right\}$$  will occur at a maximal (with respect to inclusion) element of $\mfS$.
Suppose $k >1.$ If $w_i \geq 2$ for each $i \in\{2,\ldots, n\},$ the statement holds by \Cref{trivial}. If $w_i = 1$ for all $1 \leq i \leq m$ and $w_i \geq 2$ for all $m+1\leq i\leq n-1$, the statement holds by \Cref{pathWeights}. Thus we may assume that there exists a non-trivial weight at position $p_j$ for some $j \in \{1,\ldots, k\}$ such that $w_{p_j+1} = 1$. Let $p_\ell$ be the smallest such position in the weight sequence.  Note that it is possible that $\ell=1$. These assumptions imply that $p_{\ell+1} \geq p_\ell+2$ if $\ell \neq k$ and $p_{\ell-i} = p_\ell-i$ for each $i \in \{1,\ldots, \ell-1\}.$

Under these assumptions, the edge ideal of $\P$ is
    \[
        I(\P)= (x_1x_2^{w_2}, \ldots, x_{p_\ell-1}x_{p_\ell}^{w_{p_\ell}}, x_{p_\ell}x_{p_\ell+1},x_{p_\ell+1}x_{p_\ell+2}^{w_{p_\ell+2}}, \ldots, x_{n-1}x_n^{w_n})
    \]
where $w_2, \dots, w_{p_\ell} \geq 2$ and $w_{p_{\ell+2}}, \dots, w_{p_n}$ may be anything. 

By \Cref{reg} and \Cref{oneVar} we can bound $\reg (\P)$ from above and below with the following inequalities

    \begin{eqnarray*}
        \max\{ 	\reg (R/(I(\P),x_{p_\ell})) , 	\reg (R/(I(\P),x_{p_\ell+1}))\}     &\leq & 
                \reg (\P) \\
            &\leq & 
                \max\{ 	\reg (R/(I(\P),x_{p_\ell})) , 	\reg (R/(I(\P):x_{p_\ell}))+1\}.
    \end{eqnarray*}
The remainder of the proof will consist of proving that these upper and lower bounds coincide by showing that $\reg (R/(I(\P),x_{p_\ell+1}))= \reg (R/(I(\P):x_{p_\ell}))+1$.  We will then conclude by showing that \\
$\max\{ \reg (R/(I(\P),x_{p_\ell})) , 	\reg (R/(I(\P),x_{p_\ell+1}))\}$ has the desired form.

To see that $\reg (R/(I(\P),x_{p_\ell+1}))= \reg (R/(I(\P):x_{p_\ell}))+1$ note that
    \[
        (I(\P),x_{p_\ell+1}) = (x_1x_2^{w_2}, \ldots, x_{p_\ell-1}x_{p_\ell}^{w_{p_\ell}})+ (x_{p_\ell+1})+(x_{p_\ell+2}x_{p_\ell+3}^{w_{p_\ell+3}}, \ldots, x_{n-1}x_n^{w_n})
    \]
and 
    \[
        I(\P):x_{p_\ell} = (x_1x_2^{w_2}, \ldots, x_{p_\ell-1}x_{p_\ell}^{w_{p_\ell}-1}) +  (x_{p_\ell+1}) + (x_{p_\ell+2}x_{p_\ell+3}^{w_{p_\ell+3}}, \ldots, x_{n-1}x_n^{w_n}).
    \]  
Since each of these ideals is the sum of three ideals whose generators have disjoint supports, we can calculate their regularities by summing the regularities of the component ideals by \Cref{lem:disjoint}. The regularity of $(x_1x_2^{w_2}, \ldots, x_{p_\ell-1}x_{p_\ell}^{w_{p_\ell}})$ and $(x_1x_2^{w_2}, \ldots, x_{p_\ell-1}x_{p_\ell}^{w_{p_\ell}-1})$ can be calculated using \Cref{pathWeights} and since the only difference between these ideals is the exponent of $x_{p_\ell}$ it is easy to see that $\reg (R/(I(\P),x_{p_\ell+1}))= \reg (R/(I(\P):x_{p_\ell}))+1$.  Thus $\reg(\P) = \max\{ \reg (R/(I(\P),x_{p_\ell})) , 	\reg (R/(I(\P), x_{p_\ell+1}))\}$.  It remains to show that this maximum is of the form given in the statement of the theorem.  To do this we must consider several cases.

\textit{Case 1:} First suppose that $\ell = k$.  In this case, by our assumptions on $\ell$, $p_{i+1}-p_i = 1$ for all $1 \leq i \leq k-1$ and it is clear that $\displaystyle \max_{{\bf q} \in \mfS} \left\{ \sum (q_1,\ldots, q_t)+\BL{\frac{n-q_t}{3}}\BR+\BL{\frac{q_1+1}{3}}\BR \right\}$ will occur when $(q_1, \dots, q_t) = (p_1, \dots p_k)$.  Thus we wish to show that $\displaystyle \reg(\P) = \sum_{i = 1}^k w_{p_i} + \sum_{i = 1}^{k-1}\BL\frac{p_{i+1}-p_i}{3}\BR -k +\BL\frac{n-p_k}{3}\BR + \BL \frac{p_1+1}{3}\BR$.

Note that the ideals
    \[
        (I, x_{p_k}) = (x_1x_2^{w_2}, \dots, x_{p_k-2}x_{p_k-1}^{w_{p_k-1}})+(x_{p_k})+(x_{p_k+1}x_{p_k+2}, \dots, x_{n-1}x_n)
    \]
    and 
     \[
        (I, x_{p_k+1}) = (x_1x_2^{w_2}, \dots, x_{p_k-1}x_{p_k}^{w_{p_k}})+(x_{p_k+1})+( x_{p_k+2}x_{p_k+3}, \dots, x_{n-1}x_n)
    \]
are both written as the sum of the edge ideal of a weighted naturally oriented path as in \Cref{pathWeights}, the edge ideal of a path with trivial weights, and the ideal generated by a variable where the supports of all three ideals are disjoint. We can then use \Cref{rem:Ix}, \Cref{pathWeights} and \Cref{rem:pc} to obtain
  \begin{eqnarray*}
        \reg(R/(I, x_{p_k})) 
            &=& 
                \sum_{i = 1}^{k-1} w_{p_i}+ \BL \frac{p_1+1}{3} \BR -(k-1) + \BL \frac{n-p_k+1}{3} \BR\\
            &\leq& 
                \sum_{i = 1}^{k-1} w_{p_i}+ \BL \frac{p_1+1}{3} \BR -k + 1 + \BL \frac{n-p_k}{3} \BR+1\\
            &\leq& 
                \sum_{i = 1}^{k-1} w_{p_i}+ \BL \frac{p_1+1}{3} \BR -k + \BL \frac{n-p_k}{3} \BR + w_{p_k}\\
            &=& 
                \reg(R/(I, x_{p_k+1})).
    \end{eqnarray*}
    
Therefore $\displaystyle \reg(\P) = \reg(R/(I, x_{p_k+1})) = \sum_{i = 1}^{k} w_{p_i}+ \BL \frac{p_1+1}{3} \BR -k + \BL \frac{n-p_k}{3} \BR$ as desired.

\textit{Case 2:}  Now suppose that $1<\ell<k$. Then this means that $p_2 = p_1+1$ and therefore every maximal element of $\mfS$ has $p_1$ as it's first element.  In this case
    \[
        (I(\P),x_{p_\ell})= (x_1x_2^{w_2}, \ldots, x_{p_\ell-2}x_{p_\ell-1}^{w_{p_\ell-1}})+( x_{p_\ell})+(x_{p_\ell+1}x_{p_\ell+2}^{w_{p_\ell+2}}, \ldots, x_{n-1}x_n^{w_n})
    \]
is the sum of an ideal of the form of \Cref{pathWeights} and the edge ideal of a naturally oriented path with fewer than $k$ non-trivial weights.  By \Cref{pathWeights} and the induction hypothesis we have
\small{\begin{eqnarray*} 
\reg (R/(I(\P),x_{p_\ell}))	    &=&	
	        \sum_{i=1}^{\ell-1}w_{p_i}+{\BL{\frac{p_1+1}{3}}\BR}-(\ell-1)+ \max_{\substack{\textbf{q} \in \mfS \\ q_1 \geq p_\ell+2}} \left\{ \sum (q_1,\ldots, q_t) +{\BL{\frac{n-q_t}{3}}\BR}+{\BL{\frac{q_1-p_l+1}{3}}\BR}\right\}\\
	    &=&	 
	        \max_{\substack{\textbf{q} \in \mfS \\ q_1 \geq p_\ell+2}} \left\{ \sum (p_1,\ldots, p_\ell-1,q_1,\ldots, q_t) +{\BL{\frac{n-q_t}{3}}\BR}+{\BL{\frac{p_1+1}{3}}\BR} \right\}\\
		&=&	 
		    \max_{\substack{\textbf{q} \in \mfS \\ q_1 \geq p_\ell+2}} \left\{ \sum (p_1,\ldots, p_{\ell-1},q_1,\ldots, q_t) +{\BL{\frac{n-q_t}{3}}\BR}+{\BL{\frac{p_1+1}{3}}\BR} \right\},
\end{eqnarray*}}
where the equalities use the fact that $p_{\ell-1}=p_{\ell}-1$
Similarly, 
    \[
        (I(\P),x_{p_\ell+1}) = (x_1x_2^{w_2}, \ldots, x_{p_\ell-1}x_{p_\ell}^{w_{p_\ell}}) + (x_{p_\ell+1}) + (x_{p_\ell+2}x_{p_\ell+3}^{w_{p_\ell+3}}, \ldots, x_{n-1}x_n^{w_n})
    \]
is also the sum of an ideal of the type in \Cref{pathWeights} and the edge ideal of a naturally oriented path with fewer than $k$ non-trivial weights.  Thus, again by \Cref{pathWeights} and the induction hypothesis,
\small{\begin{eqnarray*}		
\reg (R/(I(\P),x_{p_l+1}))&=& 
	        \sum_{i=1}^{l}w_{p_i}+{\BL{\frac{p_1+1}{3}}\BR}-l+ \max_{\substack{\textbf{q} \in \mfS \\ q_1 \geq p_\ell+3}} \left\{ \sum (q_1,\ldots, q_t) +{\BL{\frac{n-q_t}{3}}\BR}+{\BL{\frac{q_1-p_l-1+1}{3}}\BR} \right\}\\
	   & =& 	
	        \sum_{i=1}^{l}w_{p_i}+{\BL{\frac{p_1+1}{3}}\BR}-l+ \max_{\substack{\textbf{q} \in \mfS \\ q_1 \geq p_\ell+3}} \left\{ \sum (q_1,\ldots, q_t) +{\BL{\frac{n-q_t}{3}}\BR}+{\BL{\frac{q_1-p_l}{3}}\BR} \right\}\\
		&=&  
		    \max_{\substack{\textbf{q} \in \mfS \\ q_1 \geq p_\ell+3}} \left\{ \sum (p_1,\ldots, p_l,q_1,\ldots, q_t) +{\BL{\frac{n-q_t}{3}}\BR}+{\BL{\frac{p_1+1}{3}}\BR} \right\}.
\end{eqnarray*}}

Observe that for any ${\bf q} \in \mfS,$ we must have
$q_i= p_i$ for each $1\leq i \leq \ell-1$ and either $q_\ell=p_\ell$ (which implies that $q_{\ell+1} \geq p_\ell+3$ since our choice of $\ell$ implies that $q_{\ell+1}\neq p_\ell+1$ and the definition of $\mfS$ then implies that $q_{\ell+1}\neq p_\ell+2$) or  $q_{\ell} \geq  p_\ell+2$. Hence, the maximal elements of the set $\mfS$ are contained in the set
    \[ 
         \left\{ ( p_1,\ldots, p_\ell,q_1,\ldots, q_t) ~:~\textbf{q} \in \mfS,  \\ q_1 \geq p_\ell+3 \right\}  \cup \left\{(p_1,\ldots, p_{\ell-1},q_1,\ldots, q_t) ~:~\textbf{q} \in \mfS,  \\ q_1 \geq  p_\ell+2 \right\} 
    \]
and we deduce that
    \[
        \max \{\reg (R/(I(\P),x_{p_\ell})), \reg (R/(I(\P),x_{p_\ell+1})) \}= \max_{\substack{\textbf{q} \in \mfS}} \left\{\sum(q_1,\ldots q_t) +{\BL{\frac{n-q_t}{3}}\BR}+{\BL{\frac{p_1+1}{3}}\BR}  \right\}.
    \]

\textit{Case 3:}  Suppose finally that $\ell=1.$ Then $p_2 \geq p_1+2$ by our assumptions. Then
    \[
        (I(\P),x_{p_1}) = (x_1x_2, \ldots, x_{p_1-2}x_{p_1-1}, x_{p_1},x_{p_1+1}x_{p_1+2}^{w_{p_1+2}}, \ldots, x_{n-1}x_n^{w_n})
    \]
and by \Cref{lem:disjoint} and the induction hypothesis
    \begin{eqnarray*}
        \reg (R/(I(\P),x_{p_1}))
            &=&
                {\BL{\frac{p_1-1+1}{3}}\BR} + \max_{\substack{\textbf{q} \in \mfS \\ q_1 \geq p_1+2}} \left\{ \sum (q_1,\ldots, q_t) +{\BL{\frac{n-q_t}{3}}\BR}+{\BL{\frac{q_1-p_1+1}{3}}\BR} \right\}\\
	        &=&
	            {\BL{\frac{p_1}{3}}\BR} + \max_{\substack{\textbf{q} \in \mfS \\ q_1 \geq p_1+2}} \left\{ \sum (q_1,\ldots, q_t) +{\BL{\frac{n-q_t}{3}}\BR}+{\BL{\frac{q_1-p_1+1}{3}}\BR}  \right\}.
    \end{eqnarray*}

Similarly, for the ideal 
    \[
        (I(\P),x_{p_1+1})= (x_1x_2, \ldots, x_{p_1-2}x_{p_1-1}, x_{p_1-1}x_{p_1}^{w_{p_1}}, x_{p_1+1},x_{p_1+2}x_{p_1+3}^{w_{p_1+3}}, \ldots, x_{n-1}x_n^{w_n})
    \]
again by \Cref{lem:disjoint}, \Cref{prop:path1}, and the induction hypothesis, we have
    \begin{eqnarray*}
    	\reg (R/(I(\P),x_{p_1+1}))
    	    &=& 	
    	        w_{p_1}+{\BL{\frac{p_1+1}{3}}\BR}-1+ 
    	        \max_{\substack{\textbf{q} \in \mfS \\ q_1 \geq  p_1+3}} \left\{ \sum (q_1,\ldots, q_t) +{\BL{\frac{n-q_t}{3}}\BR}+{\BL{\frac{q_1-p_1-1+1}{3}}\BR}  \right\}\\
	        &=& 
	            \max_{\substack{\textbf{q} \in \mfS \\ q_1 \geq  p_1+3}} \left\{ \sum (p_1, q_1,\ldots, q_t) +{\BL{\frac{n-q_t}{3}}\BR}+{\BL{\frac{p_1+1}{3}}\BR} \right\}.
    \end{eqnarray*}
 Since $p_2 \geq p_1+2,$ observe that  any maximal element ${\bf q}$ of $\mfS$ with $q_1=p_1$ must have $q_2 \geq p_1+3$. Thus, we get 
  \[
        \reg (R/(I(\P),x_{p_1+1}))=\max_{\substack{\textbf{q} \in \mfS \\ q_1 = p_1}} \left\{ \sum (q_1,\ldots, q_t) +{\BL{\frac{n-q_t}{3}}\BR}+{\BL{\frac{q_1+1}{3}}\BR} \right\}.
    \]
   \sloppy  Our goal is to show that the maximum of the above expressions obtained for $\reg (R/(I(\P),x_{p_1+1}))$ and $\reg (R/(I(\P),x_{p_1}))$ is of the desired form. We make the following useful observations to achieve this goal. \\\\
    If $p_2= p_1+2,$ for any $\q= (q_1, \ldots, q_t) \in \mfS$ where $q_1=p_2,$ observe that
    \small{\begin{eqnarray*}
        {\BL{\frac{p_1}{3}}\BR} + \sum \q  +{\BL{\frac{n-q_t}{3}}\BR}+{\BL{\frac{q_1-p_1+1}{3}}\BR}
       &=& {\BL{\frac{p_2-2}{3}}\BR} + \sum (p_2,q_2,\ldots, q_t) +{\BL{\frac{n-q_t}{3}}\BR}+{\BL{\frac{p_2-p_1+1}{3}}\BR}\\
            &=& 
                {\BL{\frac{p_2-2}{3}}\BR} + \sum (p_2,q_2,\ldots, q_t) +{\BL{\frac{n-q_t}{3}}\BR}+1\\
            &=& 
                \sum (p_2, q_2,\ldots, q_t) +{\BL{\frac{n-q_t}{3}}\BR}+ {\BL{\frac{p_2+1}{3}}\BR}.
    \end{eqnarray*}}
If $p_2> p_1+2,$ then for any $(q_1,\ldots, q_t) \in \mfS$ where $q_1 > p_1+2,$ we have 
\begin{eqnarray*}
      {\BL{\frac{p_1}{3}}\BR} + \sum \q  
      +{\BL{\frac{n-q_t}{3}}\BR}+{\BL{\frac{q_1-p_1+1}{3}}\BR} 
           & \leq&  
                w_{p_1}+ {\BL{\frac{p_1+1}{3}}\BR}-1+ \sum \q
            +{\BL{\frac{n-q_t}{3}}\BR}
            +{\BL{\frac{q_1-p_1}{3}}\BR}\\ 
            &=& 
                \sum (p_1, q_1,\ldots, q_t)
                +{\BL{\frac{n-q_t}{3}}\BR}
                +{\BL{\frac{p_1+1}{3}}\BR}
    \end{eqnarray*}
    since $w_{p_1}\geq 2.$\\
    Therefore, by putting all of the observations together, we see that
\small{\begin{eqnarray*}
    \reg(\P)  &=& \max\left\{\reg(R/(I(\P), x_{p_1+1})),  \reg(R/(I(\P), x_{p_1}))\right\}\\
        &=&\max\left\{ \max_{\substack{\textbf{q} \in \mfS \\ q_1 = p_1}} \left\{ \sum \q  +{\BL{\frac{n-q_t}{3}}\BR}+{\BL{\frac{q_1+1}{3}}\BR} \right\}, \max_{\substack{\textbf{q} \in \mfS \\ q_1 \geq p_2}} \left\{\BL\frac{p_1}{3}\BR +\sum \textbf{q} + \BL \frac{n-q_t}{3}\BR + \BL\frac{q_1 - p_1+1}{3}\BR \right\}\right\} \\
         &=&\max\left\{ \max_{\substack{\textbf{q} \in \mfS \\ q_1 = p_1}} \left\{ \sum \q  +{\BL{\frac{n-q_t}{3}}\BR}+{\BL{\frac{q_1+1}{3}}\BR} \right\}, \max_{\substack{\textbf{q} \in \mfS \\ q_1 = p_2}} \left\{ \sum \q  +{\BL{\frac{n-q_t}{3}}\BR}+{\BL{\frac{q_1+1}{3}}\BR} \right\} \right\} \\
        &=& \max_{\substack{\textbf{q} \in \mfS \\ q_1 = p_1, p_2}} \left\{ \sum \textbf{q} + \BL\frac{n-q_t}{3}\BR + \BL \frac{p_1+1}{3}\BR \right\}
\end{eqnarray*} 
where $q_1=p_2$ occurs only when $p_2=p_1+2.$}
	\end{proof}

\begin{corollary}\label{cor:not2apart}
Let $\P$ be a  weighted naturally oriented path on vertices $x_1,\ldots, x_n$  with the weight sequence $(p_1,\ldots, p_k).$  If $p_{i+1}-p_i \neq 2$ for each $i \in \left\{1, \ldots, k-1\right\},$ then  
    \[
        \reg (\P)= \sum_{i=1}^k w_{p_i} + \sum_{i=1}^{k-1} {\BL{\frac{p_{i+1}-p_i}{3}}\BR}+\BL{\frac{n-p_k}{3}\BR} +\BL{\frac{p_1+1}{3}\BR}-k
    \]
where $ k \ge 1.$
\end{corollary}	

\begin{proof}
    Since $p_{i+1}-p_i \neq 2$ for each $i \in \left\{1, \ldots, k-1\right\},$ the maximal element under containment of sets in $\mfS$ is $(p_1,\ldots, p_k) $ and the statement holds by \Cref{path:gen}.
\end{proof}

\begin{example}
This example serves to illustrate that the changing the positions of the nontrivial weights can change the regularity of the graph. Let $\P_8^3, \P_8^4$ be  the weighted oriented paths pictured below in \Cref{fig:PathExample3} where the two graphs each have three vertices with the same nontrivial weight values but in different positions in the graphs.
\begin{figure}[ht]
    \centering
    \includegraphics{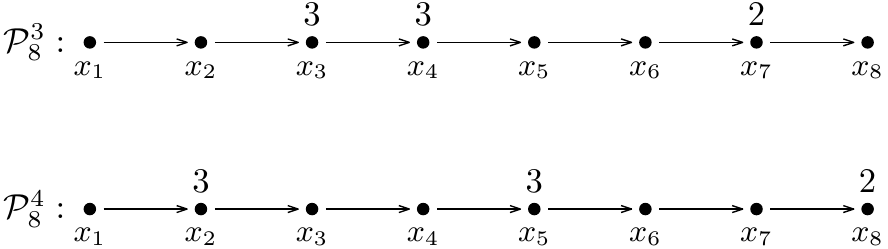}
    \caption{Two paths with the same total weights but different weight sequences}
    \label{fig:PathExample3}
\end{figure}
By using the formula in \Cref{cor:not2apart}, we can compute 
    $\reg (\P_8^3)=7$ while $\reg(\P_8^4)=8$. Here the added distance between the nontrivial weights in $\P_8^4$ results in the regularity going up by one. 
    
\end{example}

The following lemma is a consequence of the previous theorem and will be used in \Cref{cycles}. In this lemma, our objects are ideals obtained by taking the sum of an edge ideal of a  weighted naturally oriented path on $n$ vertices $x_1,\ldots, x_n$ and monomial ideal $(x_1^{w_1})$ where $w_1>1.$ Even though we introduce the term weight sequence for weighted oriented paths, we can adopt this notion for these special ideals as they can be seen as an extension of edge ideals of weighted oriented paths. Let  $I=(x_1^{w_1})+I(\P)$ be a monomial ideal where $w_1 >1$ and $\P$ is a  weighted naturally oriented path. We call $(1=p_1,p_2,\ldots, p_k)$ the weight sequence of the ideal $I$ where $(p_2, \ldots , p_k)$ is the weight sequence of $\P.$ By abusing notation, we can define the set $\mfS$ in a similar way for these type of ideals. In the next lemma we again take $\textbf{q}=(q_1,\ldots, q_t)$ to be a subsequence of $(1=p_1,p_2,\ldots, p_k)$ in $\mfS.$

\begin{lemma}\label{bigclaim3}
Let $L= (x_1^{w_{1}}, x_{1}x_{2}^{w_{2}}, \ldots, x_{n-1}x_n^{w_n})$ be an ideal in $R$  with the weight sequence $(1=p_1,\ldots, p_k).$ Then, we have 
$$ \reg (R/L)= \max_{q_1 = 1, 3} \left\{ \sum(q_1,\ldots, q_t)+\BL{\frac{n-q_t}{3}\BR}+\BL{\frac{q_1}{3}\BR}\right\}.$$
\end{lemma}

\begin{proof} 
Let $L' = (x_0x_1^{w_1-1}, x_1x_2^{w_2}, \dots, x_{n-1}x_n^{w_n})$.  Note that $L$ and $L'$ have the same polarization (up to a relabeling vertices) and thus have the same Betti numbers.  In particular $\reg(R/L) = \reg(R/L')$ and we can use \Cref{path:gen} to calculate the regularity of $R/L'$.  The ideal $L'$ is the edge ideal of a  weighted naturally oriented path on $n+1$ vertices with weight sequence $p_1+1, \dots, p_k+1$ if $w_1 \neq 2$, and with weight sequence $p_2+1, \dots, p_k+1$ if $w_1 = 2$.  We therefore have two cases to consider.  

\textit{Case 1:}  Suppose first that $w_1 \neq 2$.  Then by \Cref{path:gen} we have
\begin{eqnarray*}
\reg(R/L') 
    &=& 
        \max\left\{
            \begin{array}{ll}
                \displaystyle
                \max_{q_1 =3}\left\{\sum q + \BL\frac{n-q_t}{3}\BR+\BL\frac{q_1+2}{3}\BR\right\},\\
                \displaystyle
                \max_{q_1 = 1}\left\{ w_1-1 +\sum_{i = 2}^t w_{q_i}+\sum_{i = 1}^{t-1}\BL \frac{q_{i+1}-q_i}{3}\BR + \BL\frac{n-q_t}{3}\BR+\BL\frac{q_1+2}{3}\BR-t\right\}
            \end{array}
        \right\}\\
    &=& 
        \max\left\{
            \begin{array}{ll}
                \displaystyle
                \max_{q_1 =3}\left\{\sum q + \BL\frac{n-q_t}{3}\BR+\BL\frac{q_1+2}{3}\BR\right\},\\
                \displaystyle
                \max_{q_1 = 1}\left\{ \sum_{i = 1}^t w_{q_i}+\sum_{i = 1}^{t-1}\BL \frac{q_{i+1}-q_i}{3}\BR + \BL\frac{n-q_t}{3}\BR+\BL\frac{q_1-1}{3}\BR-t\right\}
            \end{array}
        \right\}\\
    &=& 
        \max\left\{\max_{q_1 =3}\left\{\sum q + \BL\frac{n-q_t}{3}\BR+\BL\frac{q_1+2}{3}\BR\right\} ,\max_{q_1 = 1}\left\{ \sum q + \BL\frac{n-q_t}{3}\BR+\BL\frac{q_1-1}{3}\BR\right\}\right\}\\
\end{eqnarray*}
Note that if $q_1 = 1$ then $\BL \frac{q_1 - 1}{3} \BR = \BL \frac{q_1}{3}\BR$ and if $q_1 = 3$ then $\BL \frac{q_1+2}{3} \BR = \BL \frac{q_1}{3}\BR$.  So 
\[
    \reg(R/L') = \max_{q_1 = 1, 3}\left\{\sum q + \BL\frac{n-q_t}{3}\BR + \BL \frac{q_1}{3}\BR\right\}.
\]

\textit{Case 2:} \sloppy  Suppose now that $w_1 = 2$.  Then $L'$ is a path of length $n+1$ with weighted sequence $p_2+1, \dots, p_k+1$.  By \Cref{path:gen} 
\[
    \reg(R/L') = \max_{q_1 = p_2, p_2+2}\left \{ \sum q + \BL \frac{n-q_t}{3}\BR + \BL \frac{q_1+2}{3}\BR\right\}
\]
If $(q_1, \dots, q_t) \in S$ with $q_1 \neq  3$ then $(1, q_1, \dots, q_t)$ is in $S$ also.  Recalling that in this case $w_1 = 2$, we have
\begin{eqnarray*}
     &&\sum_{i = 1}^t w_{q_i}+ \sum_{i = 1}^{t-1} \BL\frac{q_{i+1}-q_1}{3}\BR -t + \BL \frac{n-q_t}{3} \BR + \BL \frac{q_1+2}{3}\BR \\
    &=& \sum_{i = 1}^t w_{q_i}+ \sum_{i = 1}^{t-1} \BL\frac{q_{i+1}-q_1}{3}\BR -t + \BL \frac{n-q_t}{3} \BR + \BL \frac{q_1+2}{3}\BR + w_1-2\\
    &=& w_1+ \sum_{i = 1}^t w_{q_i}+ \sum_{i = 1}^{t-1} \BL\frac{q_{i+1}-q_1}{3}\BR -t + \BL \frac{n-q_t}{3} \BR + \BL \frac{q_1-1}{3}\BR -1\\
    &=& w_1+ \sum_{i = 1}^t w_{q_i}+ \sum_{i = 1}^{t-1} \BL\frac{q_{i+1}-q_1}{3}\BR + \BL \frac{q_1-1}{3}\BR -(t+1) + \BL \frac{n-q_t}{3} \BR + \BL\frac{1}{3}\BR.\\
\end{eqnarray*}
Further, if $q_1 = 3$ then, as before, $\BL \frac{q_1+2}{3}\BR = \BL \frac{q_1}{3}\BR$.  Therefore we have
\begin{eqnarray*}
    \reg(R/L') &=&   \max_{q_1 = p_2, p_2+2}\left \{ \sum q + \BL \frac{n-q_t}{3}\BR + \BL \frac{q_1+2}{3}\BR\right\}\\
    &=& \max_{q_1 = 1, 3}\left \{ \sum q + \BL \frac{n-q_t}{3}\BR + \BL \frac{q_1}{3}\BR\right\}.
\end{eqnarray*}

\end{proof}

%% file: Cycle.tex
\section{Weighted Oriented Cycles}\label{cycles}

Our main theorem of this section calculates the regularity of the edge ideal of a weighted naturally oriented cycle and presents a formula similar to the one obtained for weighted naturally oriented paths in \Cref{path:gen}. We begin with an analogous setup as in the weighted naturally oriented path case.

Let $\C_n$ be a weighted naturally oriented cycle on the vertices $x_1,\ldots, x_n$ (in order) with directed edges $(x_1, x_2),$ $(x_2, x_3),$ $\ldots,(x_{n-1}, x_n), (x_n, x_1)$ such that vertex $x_i$ has weight $w_i$ for each $i \in \{ 1, \ldots, n\}.$  
If all the weights of $\C_n$  are trivial, i.e. $w_i=1$ for each $i \in \{1,\ldots, n\},$ then studying $I(\C_n)$ is the same as studying the edge ideal of the unoriented cycle $I(C_n)$ for which the regularity is known, as discussed in \Cref{rem:pc}. 
Thus we shall assume that $w_i\geq 2$ for some $i$.

\begin{definition}
	Let $1\leq p_1< p_2 \ldots <p_k\le n$ be the positions of non-trivial weights in $\C_n.$  Similar to the weight-sequence definition introduced in \Cref{p_i}, we call  $(p_1,\ldots, p_k)$ the \emph{weight  sequence} of $\C_n.$   
\end{definition}

In contrast to the path case, one has more freedom when it comes to determining the positions of non-trivial weights of a weighted oriented cycle as one can reorder the vertices of the cycle without changing the structure of the graph. 
Thus, without loss of generality, when the cycle contains at least one trivial weight, we assume that $p_1=1$ for the remainder of the paper. 

\begin{notation}\label{notcycle}
Similar to the path case, our formula for the regularity of the cycle will be given in terms of subsequences of the weight sequence in which no two consecutive entries are distance two apart. Due to the structure of the cycle we modify the set we previously called $\mfS$ as follows. 
    \[
		\dS= \{ (q_1,\ldots, q_t) \subseteq (p_1,\ldots, p_k) ~:~   q_{i+1}-q_i \neq 2 \text{ for each } i \in \{1, \ldots, t-1\} , \text{ and }q_t-q_1\neq n-2\}.
	\]

We are again using the subset symbol to denote the subsequence relation.

The formula for the regularity of the cycle will be quite similar to that of the path and will include the weight-position sum defined in \Cref{not2} where instead we take $\q\in\dS$:
    \begin{equation}\label{weightsum}
        \sum\q=\sum(q_1,\dots,q_t):=\sum_{i=1}^tw_{q_i}+\sum_{i=1}^{t-1}\BL{\frac{q_{i+1}-q_i}{3}}\BR-t.
    \end{equation}
    \end{notation}

The formula for the regularity of the weighted oriented cycle will differ from the weighted oriented path as one needs to take into account the difference between the positions of the first and last non-trivial weights for weighted oriented cycles.

\begin{example}\label{cycleExamples}
  Let $\C_{10}^1$ and $\C_{10}^2$ be weighted oriented cycles as shown in \Cref{fig:CycleExample}, with the weight sequences $(1,3,4,6,7,9)$ and $(1,2,4,6,7,9)$, respectively. Denote by $\dS_1$ and $\dS_2$ their corresponding sets as defined in \Cref{notcycle}.
  
  \begin{figure}[h!]
      \centering
      \includegraphics{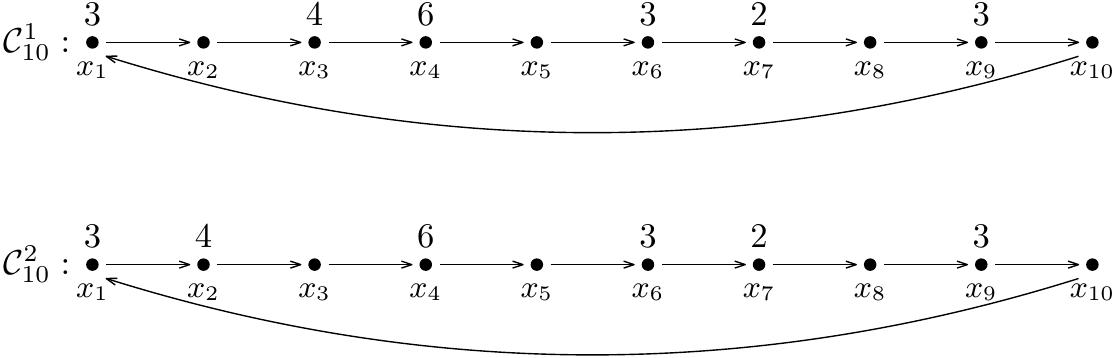}
      \caption{Two naturally oriented weighted cycles on $10$ vertices}
      \label{fig:CycleExample}
  \end{figure}
  
  Then the maximal elements with respect to inclusion of $\C_{10}^1$ in $\dS_1$ are  
    \[ 
        \{(1,4,7), (1,6,7), (3,4,7), (3,4,9), (3,6,7), (3,6,9)\}.
    \] 
    Note that each of the maximal elements satisfy either $q_1=1,$ or $q_1=3$. On the other hand, the maximal elements of $\C_{10}^2$ in $\dS_2$ are 
    \[ 
        \{(1,2,6,7), (1,4,7), (2,6,9), (4,9) \}.
    \]
    Here we see that each of the maximal elements satisfies either $q_1=1$ or $q_t=9$. Furthermore, we see how different $\dS_1$ is from $\dS_2$ although their weight sequences differ by only one entry. 
\end{example}

\begin{theorem}\label{cycle:gen}
\sloppy 	Let $\C_n$ be a weighted naturally oriented cycle on $n$-vertices $x_1,\ldots, x_n$  with the weight sequence $(p_1,\ldots, p_k).$ Then
		\[
			\reg (\C_n) = \max_{q\in\dS} \left\{ \sum(q_1,\ldots, q_t)+{\BL{\frac{n+q_1-q_t}{3}}\BR} \right\}.
		\]
\end{theorem}

Before we prove the theorem, we begin with a lemma analogous to \Cref{maximal}.

\begin{lemma}\label{cycle:maximal}
    If $\C_n$ is a  weighted naturally oriented cycle with weight sequence $\p$ and $\q\subseteq\q'\in\dS$ then
        \[
            \sum\q+{\BL{\frac{n+q_1-q_t}{3}}\BR}\leq\sum\q'+{\BL{\frac{n+q_1'-q_t'}{3}}\BR.}
        \]
    Moreover, we have 
     \begin{equation}\label{dotSmax}
        \max_{\q\in\dS}\left\{\sum\q+\BL\frac{n+q_1-q_t}{3}\BR\right\}=\max_{\substack{\q\in\dS\\q_1=1\text{ or }q_1= 3\\\text{or }q_t=n-1}}\left\{\sum\q+\BL\frac{n+q_1-q_t}{3}\BR\right\}.
    \end{equation}
\end{lemma}

\begin{proof}
    We again first consider the situation with $\q=(q_1,\dots,q_t)\subseteq(q_1,\dots,q_j,r,q_{j+1},\dots,q_t)\in\dS.$ If $q_1<r<q_t$ then the proof proceeds exactly as in the proof of \Cref{maximal} so that
        \[
            \sum\q+\BL\frac{n+q_1-q_t}{3}\BR\leq\sum\q'+\BL\frac{n+q_1-q_t}{3}\BR. 
        \]
    For the case where $r<q_1$ the proof follows from the fact that $\BL\frac{n+q_1-q_t}{3}\BR\leq\BL\frac{n+r-q_t}{3}\BR+\BL\frac{q_1-r}{3}\BR+1.$ Similarly, when $r>q_t$ we have $\BL\frac{n+q_1-q_t}{3}\BR\leq\BL\frac{n+q_1-r}{3}\BR+\BL\frac{r-q_t}{3}\BR+1.$
    
    Extending this idea, we again see that the maximum element of $\left\{\sum\q+\BL\frac{n+q_1-q_t}{3}\BR~|~\q\in\dS\right\}$ will be attained at an element of $\dS$ of maximal length. Let us now consider when $(q_1,\dots,q_t)\in\dS$ is of maximal length. Suppose $q_t\neq n-1$ and $q_1\neq 3$ then we have $(1,q_1,\dots,q_t)\in\dS$ as $q_1-1\neq 2$ and $q_t-1\neq n-2.$ Thus we see that sequence $q\in\dS$ of maximal length will satisfy $q_1=1$, $q_1=3,$ or $q_t=n-1.$
    
\end{proof}

\begin{example}

 Let $\C_{10}^1$ and $\C_{10}^2$ be weighted oriented cycles with corresponding sequence sets $\dS_1$ and $\dS_2$ as defined in \Cref{cycleExamples}. These cycles have the same number of vertices, orientation, and total weight sum with their only difference being that the weight of vertex $x_3$ in $\C_{10}^1$ is moved to vertex $x_2$ in $\C_{10}^2$. 
 
 Given $\q=(q_1,\ldots, q_t)\in\dS_i$ for $i=1,2,$ define $g(\q)=\sum\q+\BL\frac{n+q_1-q_t}{3}\BR,$ the expression introduced in \Cref{cycle:maximal}. For each of the graphs $\C_{10}^1$ and $\C_{10}^2$, the following tables provides the value of $g(\q)$ for each maximal element of $\q\in\mfS_i$ for $i=1,2$.
    \[  
        \C_{10}^1: ~\begin{array}{r|cccccc}
            \q &(1,4,7) &  (1,6,7)   & \textbf{(3,4,7)} & \textbf{(3,4,9)} & (3,6,7) & (3,6,9) \\ \hline
            g(\q) & 11 & 7 & \textbf{12} &  \textbf{12} & 9 & 10
        \end{array}\]

        \[
        \C_{10}^2: ~\begin{array}{r|cccc}
            \q & (1,2,6,7) & \textbf{(1,4,7)} & (2,6,9) & (4,9) \\ \hline
            g(\q) & 10 & \textbf{11}  &   10 &  9
        \end{array}
    \]
 Here we see that the maximum value of $\sum\q+\BL\frac{n+q_1-q_t}{3}\BR$ for $\C_{10}^1$ comes from two maximal subsequences, namely  $(3,4,7)$ and $(3,4,9)$. The maximum value of $g(\textbf{q})$ comes from the subsequence $(1,4,7)$ for $\C_{10}^2$. Using Macaulay 2 \cite{M2}, we see that $\reg(\C_{10}^1)=12$ and $\reg(\C_{10}^2)=11$, providing evidence that the regularity of the graphs is given by $g(\q)$ which, as in the path case, depends on both the weight sequence and the values of the weights.
\end{example}
     
We now prove the main theorem of this section.

\begin{proof}[Proof of \Cref{cycle:gen}]
	If $w_i \geq 2$ for all $ 1\leq i \leq n,$ then the statement holds from \Cref{trivial}. Suppose that there exists at least one $j \in \{2,\ldots, n\}$ such that $w_j=1.$ Since $k\geq 1,$ we can always find at least one pair of consecutive vertices on the cycle such that one has trivial weight, and the other has nontrivial weight.  Without loss of generality, let $x_n,$ and $x_1$ be such a pair with $w_n=1$, and $w_1 \geq 2.$

	We proceed by using a Betti splitting to calculate the regularity of the cycle. Taking $i=1$ in the statement of \Cref{cor:bettieq}  results with a Betti splitting where $J=(x_nx_1^{w_1})$ and $K=(x_1x_2^{w_2},\ldots, x_{n-2}x_{n-1}^{w_{n-1}}, x_{n-1}x_n)$. Then
	        \[
            \reg(\mathcal{C}_n)=\max\{w_1, \reg(R/K), \reg(R/(J\cap K))-1\}.
        \] 
	Our goal is to show that maximum of $w_1,$ $\reg(R/K)$, and $\reg(R/(J\cap K))-1$ is equal to
	    \[
\max_{\substack{\q\in\dS\\q_1=1 \text{ or } q_1 =3\\ \text{ or } q_t =n-1}} \left\{\sum \q + {\BL{\frac{n+q_1-q_t}{3}}\BR} \right\}.
	    \]
	
	 It can be immediately verified that  $J\cap K = J L$ where 
		\[
			L=
				\begin{cases}
					(x_2^{w_2},x_2x_3^{w_3},\ldots,x_{n-3}x_{n-2}^{w_{n-2}},x_{n-1}) &\mbox{if }w_2 \neq 1\\		
					(x_2,x_3x_4^{w_4},\ldots,x_{n-3}x_{n-2}^{w_{n-2}},x_{n-1}) &\mbox{if }w_2 =1
				  \end{cases}.
	  	\]
    By \Cref{lem:disjoint}, we have 
		\begin{eqnarray}\label{eq:cycle1}
			\reg (R/(J\cap K))-1= w_1+\reg (R/L).
		\end{eqnarray}
	Note at this point we have that $\reg(R/(J\cap K))-1\geq w_1$ and thus
	    \[
	        \reg(\C_n)=\max\{\reg(R/K),\reg(R/(J\cap K)-1\}.
        \]

We first focus on $\reg(R/(J\cap K))-1.$ By \Cref{rem:Ix}, we can ignore the single variable generators and compute $\reg (R/L)$ by \Cref{bigclaim3} when $w_2 \neq 1$, i.e. $p_2=2$, and by \Cref{path:gen} when $w_2=1$, i.e. $p_2\geq 3$. Then, by Equation \ref{eq:cycle1}, we have 
		\begin{align*}
			\reg (R/(J\cap K))-1&= w_1+
				\begin{cases}
				\displaystyle
				    \max_{\substack{
				        (q_1\dots,q_t)\in\dS\\
				        q_1\geq p_2 \text{ and } q_t\neq n-1}} 
				    \left\{\sum \q + {\BL{\frac{n-3-(q_t-1)}{3}}\BR} +{\BL{\frac{q_1-1}{3}}\BR}\right\}  & \text{if } p_2 =2\\
				\displaystyle
					\max_{\substack{
					    (q_1,\dots,q_t)\in\dS\\
					    q_1\geq p_2 \text{ and } q_t\neq n-1}} 
					\left\{\sum \q  + {\BL{\frac{n-4-(q_t-2)}{3}}\BR} +{\BL{\frac{q_1-2+1}{3}}\BR}\right\}  & \text{if }p_2>3\\
				\displaystyle
					\max_{\substack{
					    (q_1,\dots,q_t)\in\dS\\ 
					    q_1\geq p_3 \text{ and } q_t\neq n-1}} 
					\left\{\sum \q  + {\BL{\frac{n-4-(q_t-2)}{3}}\BR} +{\BL{\frac{q_1-2+1}{3}}\BR}\right\}  &\text{if }p_2 =3
				\end{cases}\\
		    &= 
		        \begin{cases}
                \displaystyle
		                \max_{\substack{
		                (q_1,\dots,q_t)\in\dS\\
		                q_1\geq p_2 \text{ and } q_t\neq n-1}}	
		                \left\{ w_1+ \sum\q+\BL\frac{n-2-q_t}{3}\BR+\BL\frac{q_t-1}{3}\BR \right\} & \text{if } p_2\neq 3\\
                \displaystyle
		                \max_{\substack{
		                    (q_1,\dots,q_t)\in\dS\\
		                    q_1\geq p_3 \text{ and } q_t\neq n-1}}	
		                \left\{ w_1 +\sum\q+\BL\frac{n-2-q_t}{3}\BR+\BL\frac{q_t-1}{3}\BR \right\} & \text{if } p_2= 3
            \end{cases}
	    \end{align*}

Note that above expressions of which the maximums are taken are identical. Furthermore, none of the $(q_1, \ldots, q_t) \in \dS$ over which the maximums are taken can have $q_1 =3$ nor can they have $q_t=n-1,$ which is equivalent to saying $(1,q_1, \ldots, q_t) \in \dS$. By expanding the weight-sum formula in \Cref{weightsum} for any $(q_1,\ldots, q_t) \in \dS$ where $q_1 \neq 3$ and $q_t \neq n-1$ (so that $(1,q_1,\dots,q_t)\in\dS$) one can verify the following equality.
		\begin{eqnarray}\label{eq:cycle4}
			w_1+\sum (q_1,\ldots, q_t)  + {\BL{\frac{n-2-q_t}{3}}\BR} +{\BL{\frac{q_1-1}{3}}\BR}= 	\sum(1,q_1,\ldots, q_t) +{\BL{\frac{n+1-q_t}{3}}\BR}
		\end{eqnarray}
Thus, we can simplify $\reg (R/(J\cap K))-1$ to the following desired form
		\[
			\reg (R/(J\cap K))-1 = \max_{\substack{\q\in\dS\\q_1=1}} \left\{\sum \q + {\BL{\frac{n+1-q_t}{3}}\BR} \right\}.
		\]

 We now consider $\reg(R/K)$. Since $K$ is the edge ideal of a weighted oriented path on $n$ vertices with the weight sequence $(p_2,\ldots, p_k),$ by \Cref{path:gen}, we get
		\begin{equation}\label{RmodK}
			\reg (R/K)= \max_{\substack{\q\in\dot{\text{S}}\\q_1 \geq p_2}} \left\{\sum (q_1,\ldots, q_t) + {\BL{\frac{n-q_t}{3}}\BR} +{\BL{\frac{q_1+1}{3}}\BR}\right\}
		\end{equation} 
    where $q_t = n-1$ is possible.

    Next we show that maximums of the expressions computing $\reg(R/K)$ and $\reg(R/(J \cap K))-1$ yield the desired form as in the statement of the theorem.  
    In order to do that we consider the following expressions
        \[
            \sum\q + \BL\frac{n+q_1-q_t}{3}\BR \text{ with }  q_1=1,\text{ and }
        \]
        \[
            \sum\q + \BL\frac{n-q_t}{3}\BR+ \BL\frac{q_1+1}{3}\BR  \text{ with }  q_1\geq p_2
        \]
   where $\q \in \dS$ and $q_t=n-1$ is possible in the second form. Again note that the first expression is already of the desired form.
   If $q_1\geq p_2$ but $q_t\neq n-1$ then it is possible that  $(1,q_1,\dots,q_t)\in\dS$ as long as $p_2\neq 3$. In this case, Case 2(a) below, we can directly compare the two expressions and show that the larger of the two is of the desired form. 
   In each of the other cases where the expressions are not directly comparable (Case 1 being $q_t=n-1$ and Case 2(b) being $p_2=3$ with $q_t\neq n-1$) we show that the second expression can also be written in the desired form.

    \textit{Case 1:} Suppose $\q\in\dS$ such that $q_t=n-1$. In this case, we see that $\displaystyle \BL\frac{n+q_1-q_t}{3}\BR=\BL{\frac{q_1+1}{3}}\BR$ and $\displaystyle \BL\frac{n-q_t}{3}\BR=0$ giving us
        \[
            \sum\q+\BL\frac{n-q_t}{3}\BR+\BL\frac{q_1+1}{3}\BR=\sum\q+\BL\frac{n+q_1-q_t}{3}\BR.
        \]

   \textit{Case 2:}  Suppose $\q\in\dS$ such that $q_t\neq n-1$.  
   
   \textit{Case 2(a):} We first consider the subcase $p_2\neq 3$.  
   
Note that  $(q_1,\ldots, q_t) \in \dS$ where $q_1\geq p_2 $ and $q_t \neq n-1$ is equivalent to $(1,q_1,\ldots, q_t) \in \dS$ and
\begin{eqnarray*}
\sum\q + \BL\frac{n-q_t}{3}\BR+\BL\frac{q_1+1}{3}\BR  &\leq&   \sum\q + \BL\frac{n-q_t+1}{3}\BR+\BL\frac{q_1+2}{3}\BR  \\
&=& \sum\q + \BL\frac{n+1-q_t}{3}\BR+\BL\frac{q_1-1}{3}\BR +1 \\
&\leq&  \sum\q + \BL\frac{n+1-q_t}{3}\BR+\BL\frac{q_1-1}{3}\BR +w_1-1\\
&=&  \sum( 1, q_1,\ldots, q_t) + \BL\frac{n+1-q_t}{3}\BR
\end{eqnarray*}   
   
 where the last inequality follows from the assumption that $w_1 \geq 2.$
 
    \textit{Case 2(b):} We next consider the subcase $p_2 =3$.  Then, for $(q_1,\ldots, q_t) \in \dS$ where $q_1 = p_2=3,$ we have
    \begin{eqnarray*}
			\sum (q_1,\ldots, q_t) + {\BL{\frac{n-q_t}{3}}\BR} +{\BL{\frac{q_1+1}{3}}\BR} = \sum (q_1,\ldots, q_t) + {\BL{\frac{n+3-q_t}{3}}\BR}.
		\end{eqnarray*}
	
All the cases are now completed and we conclude the following
{\small\begin{align*}
    \reg (\C_n) 
        &= \max \{\reg(R/K), \reg (R/(J\cap K))-1  \}\\
        &= \max  \left\{  \max_{\substack{\q\in\dS\\q_1=1}} \left\{\sum \q + {\BL{\frac{n+1-q_t}{3}}\BR} \right\} , \max_{\substack{q\in\dot{\text{S}}\\q_1 \geq p_2}} \left\{\sum (q_1,\ldots, q_t) + {\BL{\frac{n-q_t}{3}}\BR} +{\BL{\frac{q_1+1}{3}}\BR}\right\}  \right\} \\
        &= \max  \left\{  \max_{\substack{\q\in\dS\\q_1=1}} \left\{\sum \q + {\BL{\frac{n+1-q_t}{3}}\BR} \right\} , 
        \max_{\substack{\q\in\dS\\q_t=n-1}} \left\{\sum \q + {\BL{\frac{n+q_1-q_t}{3}}\BR} \right\}, \max_{\substack{\q\in\dS\\q_1=3\\ q_t \neq n-1}} \left\{\sum \q + {\BL{\frac{n+3-q_t}{3}}\BR} \right\}  \right\} \\
&= \max_{\substack{\q\in\dS\\q_1=1 \text{ or } q_1 =3\\ \text{ or } q_t =n-1}} \left\{\sum \q + {\BL{\frac{n+q_1-q_t}{3}}\BR} \right\},
\end{align*} } 
 thus completing the proof.
	\end{proof}

Computer experiments suggest that it is a difficult task to provide a closed formula for the regularity and projective dimension of the edge ideal for an arbitrary weighted oriented graph. All evidence indicates that the positions of the non-trivial weights and the orientation of the graph play essential roles in obtaining formulas for these invariants. Even if we restrict to weighted oriented paths and cycles, but allow any orientation, the problem remains a difficult one.